\documentclass{article}
\usepackage{amsthm}
\usepackage{amsmath}
\usepackage{url}
\usepackage{amssymb}
\usepackage{epsfig}
\usepackage{colonequals}
\usepackage[dvips]{color}

\newcommand{\FF}{{\mathcal F}}
\renewcommand{\SS}{{\mathcal S}}
\newcommand{\TT}{{\mathcal T}}
\newcommand{\KK}{\mathcal{K}}

\newcommand{\claim}[2]{\begin{equation}\mbox{\parbox{\linewidth}{{\em #2}}}\label{#1}\end{equation}}
\newtheorem{theorem}{Theorem}[section]
\newtheorem{corollary}[theorem]{Corollary}
\newtheorem{lemma}[theorem]{Lemma}
\newtheorem{observation}[theorem]{Observation}
\newtheorem{conjecture}[theorem]{Conjecture}
\newtheorem{proposition}[theorem]{Proposition}
\newcount\claimno
\def\newclaim#1#2{
   \global\advance\claimno by 1\relax
   \bigskip\noindent\rlap{\rm(\the\claimno)}\ignorespaces
   \global\expandafter\edef\csname CLAIMLABEL#1\endcsname{\the\claimno}\relax
   \hangindent=33pt\hskip30pt{\sl#2}\bigskip}
\def\mylabel#1{{\label{#1}}}
\def\junk#1{}

\newenvironment{subproof}{%
  \begin{proof}[Subproof]%
}{%
  \end{proof}%
}
\begin{document}
\title{Three-coloring triangle-free graphs on surfaces V. Coloring planar graphs with distant anomalies\thanks{A preprint of an earlier version of this paper appeared (under different title) as~\cite{dkt}.}}
\author{%
     Zden\v{e}k Dvo\v{r}\'ak\thanks{Computer Science Institute (CSI) of Charles University,
           Malostransk{\'e} n{\'a}m{\v e}st{\'\i} 25, 118 00 Prague, 
           Czech Republic. E-mail: {\tt rakdver@iuuk.mff.cuni.cz}.
	   Supported by the Center of Excellence -- Inst. for Theor. Comp. Sci., Prague, project P202/12/G061 of Czech Science Foundation and by
	   project LH12095 (New combinatorial algorithms - decompositions, parameterization, efficient solutions) of Czech Ministry of Education.}
 \and
     Daniel Kr{\'a}l'\thanks{Faculty of Informatics,
     Masaryk University, Botanick\'a 68A, 602 00 Brno, Czech Republic, and
     Mathematics Institute, DIMAP and Department of Computer Science, University
     of Warwick, Coventry CV4 7AL, UK. E-mail: {\tt dkral@fi.muni.cz}}
 \and
        Robin Thomas\thanks{School of Mathematics, 
        Georgia Institute of Technology, Atlanta, GA 30332. 
        E-mail: {\tt thomas@math.gatech.edu}.
        Partially supported by NSF Grants No.~DMS-0739366 and DMS-1202640.}
}
\date{February 14, 2020}
\maketitle
\begin{abstract}
We settle a problem of Havel by showing 
that there exists an absolute constant $d$ such that
if $G$ is a planar graph in which every two distinct triangles
are at distance at least $d$, then $G$ is $3$-colorable.  
In fact, we prove a more general theorem.  Let $G$ be a planar graph,
and let ${\cal H}$ be a set of connected subgraphs of $G$, each of
bounded size, such that every two distinct members of ${\cal H}$ are
at least a specified distance apart and all triangles of $G$ are contained
in $\bigcup{\cal H}$. We give a sufficient condition
for the existence of a $3$-coloring $\phi$ of $G$ such that for every
$H\in\cal H$ the restriction of $\phi$ to $H$ is constrained in a
specified way.
\end{abstract}

\section{Introduction}

This paper is a part of a series aimed at studying the $3$-colorability
of graphs on a fixed surface that are either triangle-free, or have their
triangles restricted in some way.
Here, we are concerned with $3$-coloring planar graphs. 
All {\em graphs} in this paper are finite and simple; that is,
have no loops or multiple edges.  All \emph{colorings} that we consider are proper, assigning
different colors to adjacent vertices.  The following is a classical theorem of
Gr\"otzsch~\cite{grotzsch1959}.

\begin{theorem}
\mylabel{grotzsch}
Every triangle-free planar graph is $3$-colorable.
\end{theorem}

There is a long history of generalizations that extend the theorem
to classes of graphs that include triangles.  An easy modification of Gr\"otzsch' proof
shows that every planar graph with at most one triangle is $3$-colorable.
Even more is true---every planar graph with at most three triangles is $3$-colorable.
This was first claimed by Gr\" unbaum~\cite{grunbaum}, however his proof contains
an error.  This error was fixed by Aksionov~\cite{aksenov} and later Borodin~\cite{bornew}
gave another proof.  There are infinitely many $4$-critical planar graphs with four triangles,
but they were recently completely characterized by Borodin et al.~\cite{4c4t}.

As another direction of research, Gr\" unbaum~\cite{grunbaum} conjectured that every
planar graph with no intersecting triangles is $3$-colorable.  This was disproved by Havel~\cite{conj-havel},
who formulated a more cautious question whether there exists a constant $d$ such that every planar graph
such that the distance between every two triangles is at least $d$ is $3$-colorable.  In~\cite{havel-zbarv},
Havel shows that if such a constant $d$ exists, then $d\ge 3$, and Aksionov and Mel'nikov~\cite{aksmel} improved this bound to $d\ge 4$.
Borodin~\cite{BorIrred} constructed a family of graphs that suggests
that it may not be possible to obtain a positive answer to Havel's
question using local reductions only.

The answer to Havel's question is known to be positive under various additional conditions (e.g., no $5$-cycles~\cite{fartr1},
no $5$-cycles adjacent to triangles~\cite{fartr2}, a distance constraint on $4$-cycles~\cite{le4far}),
see the on-line survey of Montassier~\cite{montasweb} for a more complete list.
The purpose of this paper is to describe a solution to Havel's problem.

\begin{theorem}\mylabel{havel}
There exists an absolute constant $d$ such that if $G$ is a planar
graph and every two distinct triangles in $G$ are at distance at least $d$,
then $G$ is $3$-colorable.
\end{theorem}

Let us remark that our proof gives an explicit upper bound on the constant $d$ of Theorem~\ref{havel},
which however is quite large (roughly $10^{100}$), especially compared to the aforementioned lower bounds.

A natural extension of Havel's question is whether instead of triangles, we could allow other kinds of distant anomalies,
such as $3$-colorable subgraphs containing several triangles (the simplest one being a diamond, that is, $K_4$ without
an edge) or even more strongly, prescribing specific colorings of some distant subgraphs.  Similar questions have been studied
for other graph classes.  For example, Albertson~\cite{Alb98} proved that if $S$ is a set of vertices in a planar graph~$G$ that are
precolored with colors $1,\ldots, 5$ and are at distance at least $4$ from each other, then the precoloring of $S$ can be extended to a 5-coloring of $G$.
Furthermore, using the results of the third paper of this series~\cite{trfree3}, it is easy to see that any precoloring of sufficiently
distant vertices of a planar graph $G$ of girth at least $5$ can be extended to a $3$-coloring of $G$.  We can even precolor
larger connected subgraphs, as long as these precolorings can be extended locally to the vertices of $G$ at some bounded distance from
the precolored subgraphs.  Both for $5$-coloring planar graphs and $3$-coloring planar graphs of girth
at least five this follows from the fact that the corresponding critical graphs satisfy a certain isoperimetric
inequality~\cite{PosThoHyperb}.

The situation is somewhat more complicated for graphs of girth four.  Firstly, as we will discuss in Section~\ref{sec-cyl},
there is a global constraint on $3$-colorings of plane graphs based on winding number, which implies that in graphs with
almost all faces of length four, precoloring a subgraph may give restrictions on possible colorings of distant parts of the graph.
For example, if we prescribed specific colorings of the triangles in Theorem~\ref{havel}, the resulting claim would be false, even though
such precolorings extend locally.  Secondly, non-facial (separating) $4$-cycles are problematic as well and they need to be treated
with care in many of the results of this series, see e.g. Theorem~\ref{mainlemma} below.  Specifically, we cannot replace triangles
in Theorem~\ref{havel} by diamonds, even though this seems viable when considering only the winding number argument, as shown
by the class of graphs (with many separating $4$-cycles) constructed by Thomas and Walls~\cite{tw-klein}.

Thus, in our second result, we only deal with graphs without separating $4$-cycles, and we need to allow certain flexibility in
the prescribed colorings of distant subgraphs.  The exact formulation of the result (Theorem~\ref{thm-supernosep4}) is somewhat
technical, and we postpone it till Section~\ref{sec-havel}.  Here, let us give just a special case covering several interesting
kinds of anomalies.  The \emph{pattern} of a $3$-coloring $\psi$
is the set $\{\psi^{-1}(1),\psi^{-1}(2),\psi^{-1}(3)\}$.  That is, two $3$-colorings have the same pattern if they only
differ by a permutation of colors.

\begin{theorem}\label{thm-nosep4sp}
There exists an absolute constant $d\ge 2$ with the following property.
Let $G$ be a plane graph without separating $4$-cycles.  Let $S_1$ be a set of vertices of $G$.  
Let $S_2$ be a set of $(\le\!5)$-cycles of $G$.  Let $S_3$ be a set of vertices of $G$ of degree at most $4$.
For each $v\in S_1\cup S_3$, let $c_v\in \{1,2,3\}$ be a color.  For each $K\in S_2$, let $\psi_K$
be a $3$-coloring of $K$.  Suppose that the distance between any two vertices or subgraphs belonging to $S_1\cup S_2\cup S_3$ is at least $d$.
If all triangles in $G$ belong to $S_2$, then $G$ has a $3$-coloring $\varphi$ such that
\begin{itemize}
\item $\varphi(v)=c_v$ for every $v\in S_1$,
\item $\varphi$ has the same pattern on $K$ as $\psi_K$ for every $K\in S_2$, and
\item $\varphi(u)=c_v$ for every neighbor $u$ of a vertex $v\in S_3$.
\end{itemize}
\end{theorem}

Let us remark that forbidding separating $4$-cycles is necessary when the anomalies $S_2$ (except for triangles) and $S_3$ are considered,
as shown by simple variations of the construction of Thomas and Walls~\cite{tw-klein}.
On the other hand, there does not appear to be any principal reason to exclude $4$-cycles when only precolored single vertices are allowed.

\begin{conjecture}\label{conj-farsv}
There exists an absolute constant $d\ge 2$ with the following property.
Let $G$ be a plane triangle-free graph, let $S$ be a set of vertices of $G$ and let $\psi:S\to \{1,2,3\}$
be an arbitrary function.  If the distance between every two vertices of $S$ is
at least $d$, then $\psi$ extends to a $3$-coloring of $G$.
\end{conjecture}

In Theorem~\ref{thm-supernosep4}, we show that Conjecture~\ref{conj-farsv} is implied by the following seemingly simpler statement.

\begin{conjecture}\label{conj-fourext}
There exists an absolute constant $d\ge 2$ with the following property.
Let $G$ be a plane triangle-free graph, let $C$ be a $4$-cycle bounding a face of $G$
and let $v$ be a vertex of $G$.  Let $\psi$ be a $3$-coloring of $C+v$.
If the distance between $C$ and $v$ is at least $d$, then $\psi$ extends to a $3$-coloring of $G$.
\end{conjecture}

If an $n$-vertex planar triangle-free graph $G$ has bounded maximum degree, then we can select a subset $S_1$ of its vertices of size $\Omega(n)$ such that
the distance between any two of vertices of $S_1$ is at least $d$.  If $G$ does not contain separating $4$-cycles, then by Theorem~\ref{thm-nosep4sp}, we can $3$-color $G$ so that
all vertices of $S_1$ have prescribed colors.  By choosing the colors of vertices in $S_1$,
we obtain exponentially many 3-colorings of $G$.  This solves a special case of a conjecture of Thomassen~\cite{thom-many} that all triangle-free planar graphs have exponentially
many 3-colorings.

\begin{corollary}\label{cor-expon}
For every $k\ge 0$, there exists $c>1$ such that every planar triangle-free graph $G$ of maximum degree at most $k$ and without
separating $4$-cycles has at least $c^{|V(G)|}$ 3-colorings.
\end{corollary}

While the current paper was undergoing review and revisions, Conjecture~\ref{conj-fourext} was confirmed to be true by Dvo\v{r}\'ak and Lidick\'y~\cite{cylgen-part2}.
Consequently, Conjecture~\ref{conj-farsv} is true as well, and in Corollary~\ref{cor-expon}, the assumption that there are no separating $4$-cycles can be dropped.

The rest of the paper is structured as follows.  In the next section, we state several previous results which we
need in the proofs.  In Section~\ref{sec-collaps}, we study the structure of graphs where no $4$-faces can
be collapsed without decreasing distances between anomalies, showing that they contain long cylindrical quadrangulated subgraphs.
In Section~\ref{sec-cyl}, we study the colorings of such cylindrical subgraphs.
Finally, in Section~\ref{sec-havel}, we prove a statement generalizing Theorems~\ref{havel} and \ref{thm-nosep4sp}.

\subsubsection*{Proof outline}

Let us finish the introduction by describing the main ideas of the proof of Theorem~\ref{havel}.

To deal with the aforementioned problems with separating $4$-cycles,
as well as with other technicalities arising in the argument,
we are actually going to prove a stronger result:  In the situation of Theorem~\ref{havel}, if either $C$ is a $4$-cycle in $G$,
or a $5$-cycle in $G$ disjoint from all triangles, and $\psi$ is a $3$-coloring of $C$,
then $\psi$ extends to a $3$-coloring of $G$.  Then we can without loss of generality assume $G$ has no separating $4$-cycles:
Otherwise, $G=G_1\cup G_2$ for proper induced subgraphs $G_1$ and $G_2$ intersecting in a $4$-cycle $K$, with $C\subset G_1$, and we can use induction to
first extend $\psi$ to a $3$-coloring of $G_1$, then extend the resulting coloring of $K$ to $G_2$.

Suppose now for a contradiction $G$ is a counterexample with $|V(G)|+|E(G)|$ minimum; clearly, the graph $G$ is connected.  Let $t$ denote the number of triangles in $G$.
We have $t\ge 2$, as otherwise $\psi$ extends to a $3$-coloring of $G$ by a result of Aksionov~\cite{aksenov},
see Lemma~\ref{lem:aksionov}.
By the main result of the previous paper in this series~\cite{trfree4}, see Theorem~\ref{mainlemma} below,
the minimality of $G$ and the fact that $G$ does not contain separating 4-cycles implies that the total length of $(\ge\!5)$-faces of $G$
is at most $\eta t$, for a constant $\eta\ll d$.
Since $G$ is connected, $t\ge 2$, and every two triangles in $G$ are at distance at least $d$ from each other,
observe that for some triangle $T\subset G$, there exist integers $a\le b<d/2$ such that $b-a=\Omega(d/\eta)$,
all faces of $G$ whose distance from $T$ is between $a$ and $b$ have length $4$, the total length of $(\ge\!5)$-faces of $G$
at distance less than $a$ from $T$ is at most $2\eta$, and $C$ is at distance more than $b$ from $T$.

Let $R$ denote the part of $G$ at distance between $a$ and $b$ from $T$, and let $f$ be a $4$-face in $R$.  Let $G'$ be the graph obtained
from $G$ by identifying two vertices $v_1$ and $v_2$ that are opposite on $f$ to a single vertex $v$.
If $G'$ satisfies the assumptions of the theorem, then $\psi$ extends to a $3$-coloring of $G'$ by the minimality of $G$, and
giving $v_1$ and $v_2$ the color of $v$, we obtain a $3$-coloring of $G$ extending $\psi$.  This is a contradiction, and thus the
described identification either creates a triangle, or decreases the distance between two triangles of $G$ (one of these triangles
necessarily has to be $T$, since $f$ is at distance less than $d/2$ from $T$).  This has to be the case for every $4$-face in $R$,
and as we show in Section~\ref{sec-collaps}, this is basically only possible if $R$ contains a regular cylindrical grid $R'$
whose length is significantly larger than its circumference.

Let $C_1$ and $C_2$ be the boundary cycles of this long cylindrical grid.  In Section~\ref{sec-cyl}, we use the connection between 3-colorings
and nowhere-zero 3-flows to show that any precoloring of $C_1\cup C_2$ satisfying a certain simple constrain (winding numbers on $C_1$ and $C_2$ match)
extends to a $3$-coloring of $R'$.  This enables us to finish the argument: We cut $G$ in the middle of $R'$,
obtaining two subgraphs $H_1$ and $H_2$ with $C\subseteq H_1$.  For $i\in \{1,2\}$, we fill in the newly created face of $H_i$
by a subgraph with a face bounded by a cycle $C'_i$ of length at most five and all other faces of length four,
obtaining a plane graph $H'_i$.  By the minimality of $G$, we can extend $\psi$ to a $3$-coloring $\varphi_1$ of $H'_1$,
color $C'_2$ the same way as $\varphi_1$ colors $C'_1$, and extend this coloring to a $3$-coloring $\varphi_2$ of $H'_2$.
This is easily seen to ensure that the winding numbers on $C_1$ and $C_2$ in these colorings match.
Hence, the coloring of $C_1\cup C_2$ given by $\varphi_1$ and $\varphi_2$ extends to a $3$-coloring $\varphi_3$ of $R'$.
We can now combine the restrictions of $\varphi_1$ and $\varphi_2$ to $H_1-V(R')$ and $H_2-V(R')$ with $\varphi_3$
to obtain a $3$-coloring of $G$ extending $\psi$.

In the more general setting of Theorem~\ref{thm-nosep4sp}, there are further complications arising from the fact that we need to avoid
creating separating 4-cycles (or at least, creating separating 4-cycles too close to the anomalies) and that we need to handle the case
there is only one anomaly, essentially proving the analogue of Lemma~\ref{lem:aksionov} for a graph with one anomaly sufficiently far away
from a precolored $(\le\!5)$-cycle.

\section{Previous results}

We use the following lemma of Aksionov~\cite{aksenov}.

\begin{lemma}
\mylabel{lem:aksionov}
Let $G$ be a plane graph with at most one triangle, and let $C$ be either
the null graph or a facial cycle of $G$ of length at most five.
If $C$ has length five and $G$ contains a triangle $T$, also assume that 
$C$ and $T$ are edge-disjoint.
Then every $3$-coloring of $C$ extends to a $3$-coloring of $G$.
\end{lemma}

We also need several results from previous papers of this series.
Let $G$ be a graph and $C$ its subgraph.
We say that $G$ is \emph{$C$-critical} if $G\neq C$ and for every proper subgraph $G'$ of $G$ that includes $C$,
there exists a $3$-coloring of $C$ that extends to a $3$-coloring of $G'$, but does not extend to a $3$-coloring of $G$.
The following claim is a special case of the general form of the main result of~\cite{trfree4} (Theorem~4.1).

\begin{theorem}\mylabel{mainlemma}
There exists an absolute constant $\eta$ with the following property.
Let $G$ be a plane graph and $Z$ a (not necessarily connected) subgraph of $G$ such that all
triangles and all separating $4$-cycles in $G$ are contained in $Z$.
If $G$ is $Z$-critical, then $\sum|f|\le \eta|V(Z)|$, where the summation is over all faces $f$ of $G$ of length at least five.
\end{theorem}

The following is a simple consequence of Corollary~5.3 of~\cite{trfree4}.
\begin{lemma}\label{lemma-extdisk}
Let $G$ be a triangle-free plane graph with the outer face $f_0$ bounded by a cycle and with another face $f$ bounded by a cycle of length at least $|f_0|-1$.
If every cycle separating $f_0$ from $f$ in $G$ has length at least $|f_0|-1$, then
every $3$-coloring of the cycle bounding $f_0$ extends to a $3$-coloring of $G$.
\end{lemma}

Finally, let us state a basic property of critical graphs.

\begin{proposition}\label{prop-sg}
Let $G$ be a graph and $C$ its subgraph such that $G$ is $C$-critical.
If $G=G_1\cup G_2$, $C\subseteq G_1$ and $G_2\not\subseteq G_1$, then $G_2$ is $(G_1\cap G_2)$-critical.
\end{proposition}

\section{Structure of graphs without collapsible $4$-faces}\label{sec-collaps}

Essentially all papers dealing with $3$-colorability of triangle-free planar graphs
first eliminate $4$-faces by identifying their opposite vertices, thus reducing the
problem to graphs of girth $5$.  However, this reduction might decrease distances
in the resulting graph, which constrains its applicability for the problems we consider.  In this section, we give
a structural result on graphs in that no $4$-face can be reduced.

Let $F$ be a cycle in a graph $G$, and let $S\subseteq V(G)$.
We say that the cycle $F$ is {\em $S$-tight} if $F$ has length four and 
the vertices of $F$ can be numbered $v_1,v_2,v_3,v_4$ in order
such that for some integer $t\ge0$ the vertices
$v_1,v_2$ are at distance exactly $t$ from $S$, and the
vertices $v_3,v_4$ are at distance exactly $t+1$ from $S$.
We say that a face is $S$-tight if it is bounded by an $S$-tight cycle.

A triple $(G,\SS,C)$ is a \emph{scene} if $G$ is a connected plane graph,
$\SS$ is a family of non-empty subsets of $V(G)$ each of which induces a connected subgraph of $G$,
and $C$ is either the null graph $\varnothing$ or a cycle of length at most five bounding the outer face of $G$.
For a positive integer $d$, the scene is \emph{$d$-distant} if for all distinct $S,S'\in \SS$, the distance
between $S$ and $S'$ in $G$ is at least $d$.

\begin{lemma}
\mylabel{lem:distcrit}
Let $d\ge1$ be an integer and let $(G,\SS,C)$ be a $2d$-distant scene.
Let $F$ be a cycle in $G$ of length four and assume that for each pair $u,v$ of diagonally opposite vertices
of $F$, two distinct sets in $\SS$ are at distance
at most $2d-1$ in the graph obtained from $G$ by identifying $u$ and $v$.
Then there exists a unique set $S_0\in\SS$ at distance at most $d-1$ from $F$.
Furthermore, $F$ is $S_0$-tight.
\end{lemma}
\begin{proof}
Let the vertices of $F$ be $v_1,v_2,v_3,v_4$ in order.
By hypothesis there exist sets $S_1,S_2,S_3,S_4\in\SS$, where
$S_i$ is at distance $d_i$ from $v_i$, such that
$S_1\ne S_3$, $S_2\ne S_4$, $d_1+d_3\le 2d-1$, and $d_2+d_4\le 2d-1$.
From the symmetry we may assume that $d_1\le d-1$ and $d_2\le d-1$.
The distance between $S_1$ and $S_2$ is at most $d_1+d_2+1\le 2d-1$, and thus
$S_1=S_2$.  Let us set $S_0=S_1$.  If any $S\in\SS$ is at distance
at most $d-1$ from $F$, then the distance between $S$ and $S_0$ is
at most $2(d-1)+1<2d$, and thus $S=S_0$.  It follows that $S_0$ is the unique
element of $\SS$ at distance at most $d-1$ from $F$.

Note that $S_4\ne S_2=S_1$, and hence $d_1+d_4+1\ge2d$, because $S_1$ and $S_4$
are at distance at least $2d$.
This and the inequality $d_2+d_4\le 2d-1$ imply that $d_1\ge d_2$.
But there is a symmetry between $d_1$ and $d_2$, and hence an analogous
argument shows that $d_1\le d_2$.
Thus for $t:=d_1=d_2$ the vertices $v_1,v_2$ are both at distance $t$ from
$S_0=S_1=S_2$.  If $v_4$ were at distance $t$ or less from $S_0$, then $S_0$ and $S_4$
would be at distance at most $t+d_4=d_2+d_4\le 2d-1$, a contradiction.
The same holds for $v_3$ by symmetry, and hence $v_3$ and $v_4$ are
at distance $t+1$ from $S_0$; hence, $F$ is $S_0$-tight.
\end{proof}

We often use the following observation on vertices only incident with tight
faces.
\begin{observation}\label{obs:alter}
Let $(G,\SS,C)$ be a distant scene and let $v\in V(G)$ be a vertex such that for some $S\in\SS$,
every face incident with $v$ is $S$-tight.  Let $t$ be the distance between $v$ and $S$.
Then $v$ has even degree, and in the clockwise ordering of the neighbors of $v$ in the
drawing of $G$, every second neighbor is at distance exactly $t$ from $S$, while
every other neighbor is at distance $t-1$ or $t+1$ from $S$.
\end{observation}

Let $G$ be a graph, let $S\subseteq V(G)$ and let $K$ be a cycle
in $G$. We say that $K$ is {\em equidistant} from $S$ if for some integer
$t\ge0$, every vertex of $K$ is at distance exactly $t$ from $S$.
We will also say that $K$ is equidistant from $S$ at distance $t$.

We say that a plane graph $H$ is a \emph{cylindrical quadrangulation with boundary faces $f_1$ and $f_2$}
if the distinct faces $f_1$ and $f_2$ of $H$ are bounded by cycles and all other faces of $H$ have length four.
The union of the cycles bounding $f_1$ and $f_2$ is called the \emph{boundary} of $H$.
The cylindrical quadrangulation $H$ is a \emph{joint} if $|f_1|=|f_2|$, every cycle of $H$ separating
$f_1$ from $f_2$ has length at least $|f_1|$ and the distance between $f_1$ and $f_2$ in $H$ is at least $4|f_1|$.
If $H$ appears as a subgraph of another plane graph $G$, we say that the appearance is \emph{clean} if
every face of $H$ except for $f_1$ and $f_2$ is also a face of $G$.
An \emph{$r\times s$ cylindrical grid} is the Cartesian product of a path with $r$ vertices and a cycle of length $s$.

Let $(G,\SS,C)$ be a scene, $R$ a cycle in $G$, and $S\in\SS$ a set disjoint from $R$.
Removing $R$ splits the plane into two open sets, and since $G[S]$ is connected, $S$ is contained
in one of them; let $\Omega_S(R)$ denote the other one.
We say $S$ is \emph{tightly isolated by $R$} if $R$ is an equidistant cycle of length $s\ge 3$ at some distance $d_0\ge 1$ from $S$,
and for $d_1=d_0+2(s-2)(s+3)$, letting $V_G(S,R)$ be the set of vertices of $G$ at distance at most $d_1$ from $S$ that
are drawn in the closure of $\Omega_S(R)$, every face of $G$ drawn in $\Omega_S(R)$ and incident with a vertex of $V_G(S,R)$ is $S$-tight.

\begin{lemma}
\mylabel{lem:cylinder}
Let $(G,\{S\},\varnothing)$ be a scene.  If $S$ is tightly isolated by a cycle $R_0$ in $G$
and every vertex of $V_G(S,R_0)$ has degree at least three,
then $G$ contains a clean joint $H$ such that $V(H)\subseteq V_G(S,R_0)$.
\end{lemma}
\begin{proof}
Let $s=|R_0|$ and let $d_0$ be the distance between $S$ and $R_0$ in $G$.
For an integer $j$, let $d(j)=d_0+2(s-j)(s+j+1)$.  Note that $d(j)+4j=d(j-1)$ for every $j$, $d_0=d(s)$, and
every vertex of $V_G(S,R_0)$ is at distance at most $d_1=d_0+2(s-2)(s+3)=d(2)$ from $S$.
Choose the smallest integer $j\in\{3,\ldots,s\}$ for that there exists
an equidistant cycle $R$ of length $j$ at distance $t$ from $S$ such that $d_0\le t\le d(j)$ and
$R$ is drawn in the closure of $\Omega_S(R_0)$; note this implies $V(R)\subseteq V_G(S,R_0)$.
Such an integer $j$ exists, since $R_0$ satisfies the requirements for $j=s$.
Let $p\le 4j$ be the maximum integer such that $G$ contains a clean $(p+1)\times |R|$ cylindrical grid $H$ with boundary faces $f_1$ and $f_2$
as a subgraph such that $f_1$ is bounded by $R$ and $f_2$ is bounded by an equidistant cycle $K$ at distance $t+p$ from $S$,
and $f_2$ is drawn in $\Omega_S(R)$; note this implies $V(H)\subseteq V_G(S,R_0)$.
Such an integer $p$ exists, since $R$ (treated as a $1\times |R|$ cylindrical grid) satisfies the requirements for $p=0$.

We claim that $p=4j$, and thus $H$ satisfies the conclusion of the theorem.
Suppose that $p\le 4j-1$.  Note that every vertex of $G$ drawn in $\Omega_S(K)$ is at distance at least $t+p+1$ from $S$.
Observe that $K$ has no chord contained in $\Omega_S(K)$, as otherwise there exists an equidistant cycle
of length less than $j$ at distance $t+p\le t+4j-1<d(j-1)$ from $S$ contradicting the minimality of $j$.
Hence, Observation~\ref{obs:alter} implies that every vertex $v\in V(K)$ has exactly one neighbor $v'$ drawn in $\Omega_S(K)$.

Let $Z$ be the subgraph of $G$ induced by $\{v':v\in V(K)\}$; note that $V(Z)$ consists exactly
of all vertices drawn in $\Omega_S(K)$ at distance $t+p+1\le t+4j\le d(j-1)$ from $S$, and in particular $V(Z)\subset V_G(S,R_0)$.
By the assumptions of this lemma, all vertices in $V(Z)$ have degree at least three in $G$, and thus Observation~\ref{obs:alter}
implies $Z$ has minimum degree at least two.  Consequently, $Z$ contains a cycle $Z'$.
Note that $Z'$ is equidistant at distance at most $d(j-1)$ from $S$ and $|Z'|\le|V(Z)|\le |K|=j$.
By the minimality of $j$, it follows that $|Z'|=j$, and thus $|V(Z)|=|K|$.  Therefore, $v_1'\ne v_2'$ for distinct vertices $v_1,v_2\in V(K)$.
We conclude that we can extend $H$ to a clean $(p+2)\times |R|$ cylindrical grid by adding $Z'$ and the edges $vv'$ for $v\in V(K)$,
contradicting the maximality of $p$.  This finishes the proof.
\end{proof}

Next, we consider the case that some of the relevant faces are not tight, but instead are near to a short separating cycle.
A $4$-face $f$ is \emph{attached} to a cycle $R$ if the boundary cycle of $f$ and $R$ intersect in a path of length two.
Let $d_2<d_3$ and $s$ be positive integers and let $(G,\SS,C)$ be a scene.
For $S\in\SS$, we say that a cycle $R$ \emph{separates $S$ from $C$} if $C$ is not the null graph, $R\neq C$, and $S$ is drawn
in the open disk bounded by $R$ (recall that $C$ bounds the outer face of $G$).
We say that the scene is \emph{$(d_2,d_3)$-tight} if for every $S\in\SS$, every $4$-face of $G$
at distance at least $d_2$ and at most $d_3$ from $S$ is bounded by $C$, or $S$-tight,
or attached to a $(\le\!6)$-cycle separating $S$ from $C$.
An \emph{$(S,d_2,d_3)$-slice} is a subset $L$ of vertices of $G$ such that
\begin{itemize}
\item each vertex $v\in L$ is at distance at least $d_2$ and at most $d_3$ from $S$,
\item if $v\in L$ has a neighbor in $G$ not belonging to $L$, then the distance between $S$ and $v$ is either exactly $d_2$ or exactly $d_3$, and
\item $L$ contains a vertex at distance exactly $d_3-1$ from $S$.
\end{itemize}
Note that the last two conditions imply that $L$ contains
vertices at any distance $d$ from $S$ such that $d_2\le d\le d_3-1$.  The \emph{interior} $L^\circ$ of $L$ is the set of vertices
at distance at least $d_2+1$ and at most $d_3-1$ from $S$.
When the parameters are clear from the context, we call $L$ just a slice.
For a positive integer $s$, we say that a set $S\in\SS$ is \emph{$(d_2,d_3,s)$-isolated} by an $(S,d_2,d_3)$-slice $L$ if
\begin{itemize}
\item $L\cap V(C)=\emptyset$ and every vertex of $L$ has degree at least three,
\item every face of $G$ incident with a vertex of $L$ has length four, and
\item every cycle $K\subseteq G[L]$ equidistant from $S$ has length at most $s$.
\end{itemize}

\begin{lemma}
\mylabel{lem:cylsep}
Let $d_2\ge 4$ and $s\ge 3$ be integers, let $d_3=d_2+34(s-2)(s+3)+474$, and let $(G,\{S\},C)$ be a $(d_2,d_3)$-tight scene.
If $S$ is $(d_2,d_3,s)$-isolated by a slice $L$, then $G$ contains a clean joint $H$ with $V(H)\subseteq L^\circ$.
\end{lemma}
\begin{proof}
Let $\KK$ be the set of all $(\le\!6)$-cycles $K\subset G[L^\circ]$ that separate $S$ from $C$ in $G$.
For an integer $t$ such that $d_2\le t\le d_3$, let $G_t$ denote the subgraph of $G[L]$ induced by vertices at distance exactly $t$ from $S$.
By assumptions, every cycle in $G_t$ has length at most $s$.
\claim{cl-eqcyc}{If $d_2+4\le t\le d_3-4$ and $v\in V(G_t)$ is at distance at least two from every element of $\KK$,
then all faces incident with $v$ are $S$-tight and $\deg_{G_t}(v)\ge 2$.}
\begin{subproof}
Since $v\in L$, any face $f$ of $G$ incident with $v$ is a $4$-face not bounded by $C$.
Since $d_2+4\le t\le d_3-4$, if $f$ were attached to a $(\le\!6)$-cycle $K$ separating $S$ from $C$,
then we would have $K\subset G[L^\circ]$,
and thus $K$ would be an element of $\KK$ at distance at most one from $v$, contradicting the assumptions.
Since the scene is $(d_2,d_3)$-tight, we conclude every face incident with $v$ is $S$-tight.
Since $\deg_G(v)\ge 3$, Observation~\ref{obs:alter} implies $\deg_{G_t}(v)\ge 2$.
\end{subproof}

For a cycle $K\in\KK$, let $\Delta_K$ be the closed disk bounded by $K$.  For distinct $K_1, K_2\in \KK$, we write
$K_1\prec K_2$ if $K_1$ is drawn in $\Delta_{K_2}$, and we write $G_{K_1,K_2}$ for the subgraph of $G$ drawn
in $\Delta_{K_2}\setminus \Delta^\circ_{K_1}$.

\claim{cl-nofar}{Consider cycles $K_1,K_2\in \KK$ of the same length $r$ such that $K_1\prec K_2$ and no cycle $K\in \KK$
of length less than $r$ satisfies $K_1\prec K\prec K_2$.  For $i\in\{1,2\}$, let $k_i$ denote the distance between $S$ and $K_i$.
If $k_1+4r+3\le k_2\le d_3-2(s-2)(s+3)-12$, then $G$ contains a clean joint $H$ such that $V(H)\subseteq L^\circ$.}
\begin{subproof}
Note that by the assumptions of the claim, no cycle in $G_{K_1,K_2}$ that separates $K_1$ from $K_2$ has length less than $r$
and the distance between $K_1$ and $K_2$ is at least $4r$.
If $V(G_{K_1,K_2})\subseteq L^\circ$, then since $S$ is $(d_2,d_3,s)$-isolated by $L$, all faces of
$G_{K_1,K_2}$ not bounded by $K_1$ or $K_2$ have length four, and thus we can set $H=G_{K_1,K_2}$.

Therefore, assume that $G_{K_1,K_2}$ contains a vertex not in $L^\circ$; since $L$ is a slice and $G$ is connected, we conclude
$G_{K_1,K_2}\cap G[L]$ contains vertices at any distance between $k_1$ and $d_3$ from $S$.
Let $t=k_2+8$ and let $Q$ be a connected component of $G_t$ contained
in $G_{K_1,K_2}$.  Observe that every cycle $K\in \KK$ which intersects $G_{K_1,K_2}$ is at distance
at most $k_2$ from $S$ if $K\prec K_2$, and at most $k_2+3$ if $K$ intersects $K_2$, and thus its distance from $Q$ is at least two.
By (\ref{cl-eqcyc}), $Q$ has minimum degree at least two, and thus $Q$ contains a cycle $R$, necessarily of length at most $s$.
Furthermore, (\ref{cl-eqcyc}) implies every face $f$ incident with a vertex $v\in V_G(S,R)$ is $S$-tight.
By Lemma~\ref{lem:cylinder}, $G$ contains a clean joint $H$ with $V(H)\subseteq V_G(S,R)\subseteq L^\circ$, as required.
\end{subproof}

Let $b_2=d_2-1$ and $e_2=d_3-2(s-2)(s+3)-11$.  For $3\le r\le 6$, let $b_r$ and $e_r$ be chosen so that $b_{r-1}\le b_r\le e_r\le e_{r-1}$,
every cycle in $\KK$ of length $r$ is at distance either at most $b_r$ or at least $e_r$ from $S$, and subject to these
conditions, $e_r-b_r$ is as large as possible.

Consider a fixed $r\in \{3,4,5,6\}$.  If no cycle in $\KK$ has length $r$ and is at distance more than $b_{r-1}$ and less than $e_{r-1}$ from $S$, then
we have $b_r=b_{r-1}$ and $e_r=e_{r-1}$.  Otherwise, let $K_1\in \KK$ be a cycle of length $r$ whose distance $k_1$ from $S$ satisfies $b_{r-1}<k_1<e_{r-1}$
and subject to that, $k_1$ is as small as possible; and, let $K_2\in \KK$ be a cycle of length $r$ whose distance $k_2$ from $S$ satisfies $b_{r-1}<k_2<e_{r-1}$
and subject to that, $k_2$ is as large as possible.  If $k_2\ge k_1+4r+3$, then (\ref{cl-nofar}) implies that the conclusion of this lemma holds, and thus
we can assume that $k_2\le k_1+4r+2$.  Note that the distance of every cycle in $\KK$ of length $r$ from $S$ is at most $b_{r-1}$, or between $k_1$ and $k_2$ (inclusive), or at least $e_{r-1}$.
Furthermore, $(k_1-b_{r-1})+(e_{r-1}-k_2)=(e_{r-1}-b_{r-1})-(k_2-k_1)\ge (e_{r-1}-b_{r-1})-4r-2$, and
thus, considering $(b_{r-1},k_1)$ and $(k_2,e_{r-1})$ as possible choices for $(b_r,e_r)$, we have
$e_r-b_r\ge \max(k_1-b_{r-1},e_{r-1}-k_2)\ge \frac{e_{r-1}-b_{r-1}}{2} - 2r -1$.

It follows that $e_6-b_6>\frac{e_2-b_2}{16}-22=\frac{d_3-d_2-2(s-2)(s+3)-362}{16}=2(s-2)(s+3)+7$.  Let $t=b_6+5$ and let $Q$ be
a connected component of $G_t$ (note that $G_t$ is non-empty, since $L$ is a slice).
Observe the distance between $Q$ and every element of $\KK$ is at least two, and thus by (\ref{cl-eqcyc}), $Q$ has minimum degree at least two.
Consequently, $Q$ contains a cycle $R$, necessarily of length at most $s$.
Since $t+2(s-2)(s+3)\le e_6-2$, every vertex $v\in V_G(S,R)$ is at distance at least $b_6+5$ and at most $e_6-2$ from $S$,
and thus every cycle $K\in \KK$ is at distance at least two from $v$.  Consequently, (\ref{cl-eqcyc}) implies
all faces incident with $v$ are $S$-tight.  By Lemma~\ref{lem:cylinder}, $G$ contains a clean joint $H$ such that
$V(H)\subseteq V_G(S,R)\subseteq L^\circ$, as required.
\end{proof}

Let $G$ be a plane graph.  For a set $S\subseteq V(G)$, a path $P$ from a vertex $v$ to $S$ is \emph{$S$-geodesic} if $P$ is a shortest path from $v$ to $S$.
Let $B$ be an odd cycle in $G$, let $\Lambda$ be one of the two connected open subsets of the plane bounded by $B$, let $uv$ be an edge of $B$,
let $w$ be the vertex of $B$ that is farthest (as measured in $B$) from $uv$ and let $z$ be a vertex of $G$ such that either $z=w$, or $z$ does not
belong to the closure of $\Lambda$.  Let $P_u$ and $P_v$ be the paths in $B-uv$ joining $u$ and $v$, respectively, with $w$.
We say that $\Lambda$ is a \emph{$z$-petal with top $uv$} if there exists a path $Q$ in $G$ between $w$ and $z$ such that the paths $Q\cup P_u$ and $Q\cup P_v$ are $\{z\}$-geodesic.

Let $S$ be a set of vertices inducing a connected subgraph of $G$ and consider a cycle $K$ which is equidistant at some distance $t\ge 1$ from $S$.
The removal of $K$ splits the plane into two open sets, let $\Delta$ be the one containing $S$.
For each $v\in V(K)$, choose an $S$-geodesic path $P_v$.
We can choose the paths so that for every $u,v\in V(K)$, the paths $P_u$ and $P_v$
are either disjoint or intersect in a path ending in $S$.  Removing $G[S]$ and the paths
$P_v$ for $v\in V(K)$ splits $\Delta$ to several parts; for each $e\in E(K)$, let $\Delta_e$
be the one whose boundary contains $e$.  Clearly, $\Delta_e$ and $\Delta_{e'}$ are disjoint
for distinct $e,e'\in E(K)$.  We call the collection $\{\Delta_e:e\in E(K)\}$ a \emph{flower of $K$ with respect to $S$}.
Let us remark that not all elements of a flower are necessarily petals: $\Delta_e$ is a $z$-petal with top $e$ for some
$z\in S$ if and only of the boundary of $\Delta_e$ does not contain any edge of $G[S]$.

Since a petal is bounded by an odd cycle, it contains an odd face of $G$.  However, this face could in general be arbitrarily
far from $S$.  In the next lemma, we exploit the presence of $S$-tight faces to find a face of length other than four close to $S$.

\begin{lemma}\label{lemma-nearodd}
Let $d_4$ be a positive integer and let $(G,\{S\},\varnothing)$ be a scene
such that every vertex $v$ at distance exactly $d_4$ from $S$ has degree at least three and all $4$-faces incident with $v$ are $S$-tight.
For some $d\le d_4$, let $uv$ be an edge of $G$ such that both $u$ and $v$ are at distance exactly $d$ from $S$, and suppose $z\in S$ is at distance
exactly $d$ from both $u$ and $v$.  For every $z$-petal $\Delta$ with top $uv$,
there exists a face $f\subseteq \Delta$ of $G$ of length other than four at distance at most $d_4$ from~$S$.
\end{lemma}
\begin{proof}
We can assume that $\Delta$ is minimal, i.e., there is no $\Delta'\subsetneq \Delta$ such that $\Delta'$
is a $z$-petal satisfying the assumptions of the lemma.  Since $\Delta$ is bounded by an odd cycle, there exists
an odd face $f$ contained in $\Delta$.  It suffices to consider the case that the distance between $f$ and $S$ is
at least $d_4+1$.  Let $Q$ be the subgraph of $G$ induced by vertices at distance exactly $d_4$ from $S$ that are contained
in the closure of $\Delta$.  Note that $Q$ is non-empty since $G$ is connected, and can intersect the boundary of $\Delta$ only in the edge $uv$.

If $Q=uv$, then $\{u,v\}$ forms a cut in $G$ that separates the rest of the boundary of $\Delta$ from the vertices incident with $f$.
Observe that this implies that there exists a face $f'$ contained in $\Delta$ in whose boundary $u$ and $v$ appear non-consecutively.
This implies $f'$ is not $S$-tight, and thus it is not a $4$-face.  Hence, the conclusion of this lemma is satisfied.

Therefore, we can assume that $Q\neq uv$.  By Observation~\ref{obs:alter}, all vertices of $Q$ other than $u$ and $v$
have degree at least two in $Q$.  Since $uv\in E(Q)$, it follows that $Q$ contains a cycle $K$, which is equidistant at distance $d_4$ from $S$.  
Let $F=\{\Delta_e:e\in E(K)\}$ be a flower of $K$ with respect to $S$ and let $e_0$ be the unique edge of $K$ such that the closure of $\Delta_{e_0}$
contains the edge $uv$.  Note that since every vertex in the boundary of $\Delta$ is contained in an $S$-geodesic path ending in $z$,
every vertex of $K$ is at distance exactly $d_4$ from $z$, and thus we can choose $F$ so that $\Delta_e\subset\Delta$ and $\Delta_e$
is a $z$-petal for every $e\in E(K)\setminus\{e_0\}$.  Since $|F|=|K|\ge 3$, it follows that each such $z$-petal $\Delta_e$ is a proper subset
of~$\Delta$.  This contradicts the minimality of~$\Delta$.
\end{proof}

Next, we apply Theorem~\ref{mainlemma} to prove existence of sufficiently isolated anomalies in hypothetical
minimal counterexamples to Theorem~\ref{thm-nosep4sp}.  To this end, we need a few more definitions.  For $p\ge 1$,
we say that a scene $(G,\SS,C)$ is \emph{$p$-small} if every set in $\SS$ has size at most $p$.  The scene is \emph{internally triangle-free}
if for every triangle $T\neq C$ in $G$, there exists $S\in\SS$ such that $T\subseteq G[S]$.
For $S\in \SS$, a cycle $F\neq C$ in $G$ is \emph{$S$-private} if the open disk bounded by $F$ contains a vertex of $S$, but not of any other set from $\SS$.
For an integer $d\ge 1$, we say the scene has \emph{no $d$-distant private $4$-cycles}
if for every $S\in\SS$, every $S$-private $4$-cycle in $G$ is at distance less than $d$ from $S$.
We say that a $4$-cycle is \emph{$\SS$-private} if it is $S$-private for some $S\in\SS$.

Consider a face $f$ of $G$, bounded by a closed walk $v_1v_2\ldots v_m$ going clockwise around $f$.
A pair $(v_{i-1}v_iv_{i+1},f)$ for $1\le i\le m$ (where $v_0=v_m$ and $v_{m+1}=v_1$) is called an \emph{angle}
in $G$, and $v_i$ is its \emph{tip}.

\begin{lemma}\label{lemma-isolate}
For all integers $D_1\ge 2$ and $p\ge 1$ and for every function $h:\mathbb{N}\to\mathbb{N}$, there exist integers $s\ge 1$ and $D_2>D_1$
with the following property.
Let $(G,\SS,C)$ be a $(D_1,D_2)$-tight $2D_2$-distant $p$-small internally triangle-free scene with no $D_1$-distant private $4$-cycles.
If $|\SS|=1$, assume furthermore that $C$ is not the null graph and the distance between $C$ and the unique element of $\SS$ is at least $D_2-1$.

Let $Z=C\cup \bigcup_{S\in\SS} G[S]$.  If $G$ is $Z$-critical, then there exists an integer $d\ge D_1$ such that
$d+h(s)\le D_2$ and some element of $\SS$ is $(d,d+h(s),s)$-isolated.
\end{lemma}
\begin{proof}
Let $\mu=2\eta(3p+5)$, where $\eta$ is the constant from Theorem~\ref{mainlemma}, $s=\mu+6p$, and
$D_2=D_1+3+(\mu+1)(h(s)+1)$.

By removing some of the edges of $E(Z)\setminus E(C)$ from $G$ if necessary, we can assume $G$ contains no triangle other than $C$.
Since $G$ is $Z$-critical, note that Lemma~\ref{lem:aksionov} implies $\SS\neq\emptyset$
and the open disk bounded by any separating $4$-cycle in $G$ contains a vertex of $\bigcup\SS$.
If $G$ contains a non-$\SS$-private separating $4$-cycle, then let $C_0$ be such a $4$-cycle with the closed disk $\Delta_0$ bounded by $C_0$ minimal.
Otherwise, let $\Delta_0$ be the whole plane and $C_0=C$.

For each $S\in\SS$, let $\FF_S$ denote the set of $S$-private $4$-cycles $F$ in $G$
such that the open disk $\Lambda_F$ bounded by $F$ is contained in $\Delta_0$ and is inclusionwise-maximal among all $4$-cycles with this property.
We claim that for distinct $F,F'\in \FF_S$, the disks $\Lambda_F$ and $\Lambda_{F'}$ are disjoint.  Indeed, since $G$ contains at most one triangle,
the cycles $F$ and $F'$ are induced, and thus if $\Lambda_F\cap\Lambda_{F'}\neq\emptyset$, then the open disk $\Lambda_F\cup \Lambda_{F'}$ is also bounded by
an $S$-private $4$-cycle, contradicting the maximality of $\Lambda_F$ or $\Lambda_{F'}$.
Since each of the disks contains a vertex of $S$, we conclude $|\FF_S|\le |S|$.

Furthermore, for distinct $S,S'\in\SS$ and any $F\in\FF_S$ and $F'\in\FF_{S'}$, the disks $\Lambda_F$ and $\Lambda_{F'}$ are
disjoint. Indeed, since the scene has no $D_1$-distant private $4$-cycles, the distance between $S$ and $F$, and between $S'$ and $F'$,
is less than $D_1$, and since the scene is $2D_2$-distant, the cycles $F$ and $F'$ are vertex-disjoint.
Futhermore $\Lambda_F\not\subseteq \Lambda_{F'}$ since $\Lambda_F$ contains a vertex of $S$ and $F'$ is $S'$-private,
and symmetrically $\Lambda_{F'}\not\subseteq \Lambda_F$.  This implies $\Lambda_F\cap\Lambda_{F'}=\emptyset$.

Let $\SS_1\subseteq \SS$ consist of the sets $S\in\SS$ intersecting $\Delta_0$; note that $\SS_1\neq\emptyset$.
For $S\in\SS_1$, let $\Delta_S$ be the complement of $\bigcup_{F\in \FF_S} \Lambda_S$
and let $B_S$ be the subgraph of $G[S]\cup\bigcup\FF_S$ drawn in $\Delta_S\cap \Delta_0$.
Let $G_1$ be the subgraph of $G$ drawn in the subset $\Delta_1=\Delta_0\cap\bigcap_{S\in\SS_1} \Delta_S$ of the plane.
Let $Z_1=C_0\cup \bigcup_{S\in\SS_1} B_S$; Proposition~\ref{prop-sg} implies that $G_1$ is $Z_1$-critical.

If some cycle $F\in \FF_S$ is vertex-disjoint from $S$, then since $G[S]$ is connected, we conclude $\FF_S=\{F\}$ and
$B_S=F$.  Otherwise, every cycle $F\in\FF_S$ intersects $S$, and thus $G[S]\cup\bigcup\FF_S$ is connected, and either $B_S$ is connected
or every component of $B_S$ intersects $C_0$; and furthermore, $|V(B_S)|\le 3|S|\le 3p$.
Since the scene has no $D_1$-distant private $4$-cycles, every vertex of $B_S$ is at distance at most $D_1+1$ from $S$.

Since the scene is $2D_2$-distant, at most one set in $\SS_1$ is at distance at most $D_2-2$ from $C_0$.
Moreover, if $|\SS_1|=1$, then we could not have chosen $C_0$ as a non-$\SS$-private separating $4$-cycle,
and thus $C_0=C$ is at distance at least $D_2-1$ from the unique element of $\SS_1=\SS$ by the assumptions of this lemma.
Therefore, letting $\SS'_1$ consist of the sets $S\in \SS_1$ at distance at least $D_2-1$ from $C_0$,
we have $|\SS'_1|\ge |\SS_1|/2$.

A face $f$ of $G$ is \emph{poisonous} if $f\subseteq \Delta_1$ and $f$ has length at least $5$.
The construction of $G_1$ ensures that it has no separating $4$-cycles, and thus
Theorem~\ref{mainlemma} implies
\begin{equation}\label{eq-sumf}
\sum|f|\le \eta|V(Z_1)|\le \eta(3p|\SS_1|+5)\le \mu|\SS'_1|,
\end{equation}
where the summation is over all poisonous faces of $G$.
Consider $S\in \SS'_1$. We say that an angle $(xyz,f)$ in $G$ is {\em $S$-contaminated}
if $f$ is poisonous and the distance between $S$ and $y$ in $G$ is at most $D_2-1$.
Since every $S$-contaminated angle contributes at least one toward the sum in (\ref{eq-sumf}),
we deduce that there exists $S\in\SS'_1$ such that there are at most $\mu$ angles that are $S$-contaminated.
Let us fix such a set $S$.  

By the choice of $D_2$, there exists an integer $d\ge D_1+2$
such that $d+h(s)\le D_2-2$ and no angle with tip at distance at least $d$ and at most $d+h(s)$
is $S$-contaminated.  Let $L$ consist of the vertices of $G$ drawn in $\Delta_1$
at distance at least $d$ and at most $d+h(s)$ from $S$.  Observe that $L$ is an $(S,d,d+h(s))$-slice and every vertex of $L$ is
contained in the interior of $\Delta_1$,
since $C_1$ is at distance at least $D_2-1$ from $S$, every vertex of $B_S$ is at distance at most $D_1+1$
from $S$, and for $S'\in \SS_1$, the subgraph $B_{S'}$ is at distance at least $2D_2-(D_1+1)>D_2-1$ from $S$.
In particular, $L\cap V(C)=\emptyset$.  Since $G_1$ is $Z_1$-critical, every vertex of $L$ has degree at least three.
The choice of $d$ implies that every face of $G$ incident with a vertex of $L$ has length four.

Hence, it remains to argue that every cycle $K\subseteq G[L]$ equidistant from $S$ has length at most $s$.
First, observe the argument from the previous paragraph also implies that every face $f$ of $G$
contained in $\Delta_S$ and at distance less than $D_2-1$ from $S$ is contained in $\Delta_1$.
Let $f_S$ be the face of $B_S$ containing $K$, and
let $W=\{\Delta_e:e\in E(K)\}$ be a flower of $K$ in $G$ with respect to $S$.  Observe that if the closure of $\Delta_e$
does not contain any edge of the boundary of $f_S$, then $\Delta_e$ is a $z$-petal for some $z\in S$ and $\Delta_e\subset\Delta_S$.
Lemma~\ref{lemma-nearodd} applied with $d_4=D_2-2$ to the scene $(G,\{S\},\varnothing)$ implies that there exists
a face $f\subseteq \Delta_e$ of $G$ of length other than four at distance less than $D_2-1$ from $S$, and as we observed,
this implies that $f$ is contained in $\Delta$; hence, $f$ is poisonous and contributes an $S$-contaminated angle.
Consequently, all but at most $\mu$ elements of $W$ contain an edge of the boundary of $f_S$ in their closure.
Since $|V(B_S)|\le 3p$, $f_S$ has length at most $6p$, and thus $|K|=|W|\le \mu+6p=s$, as required.
\end{proof}

We can now combine the lemmas to obtain the main structural result of this section.

\begin{lemma}\label{lemma-struct}
There exists a function $f_{\ref{lemma-struct}}:\mathbb{N}^2\to\mathbb{N}$ with the following property.
Let $D_1\ge 2$ and $p\ge 1$ be integers and let $D_2=f_{\ref{lemma-struct}}(D_1,p)$.
Let $(G,\SS,C)$ be a $(D_1,D_2)$-tight $2D_2$-distant $p$-small internally triangle-free scene with no $D_1$-distant private $4$-cycles.
If $|\SS|=1$, assume furthermore that $C$ is not the null graph and the distance between $C$ and the unique element of $\SS$ is at least $D_2-1$.
Let $Z=C\cup \bigcup_{S\in\SS} G[S]$.  If $G$ is $Z$-critical, then $G$ contains a clean joint $H$ whose vertices are at distance at least
$D_1$ and at most $D_2-1$ from some element of $\SS$.  Furthermore, $H$ is vertex-disjoint from $C$.
\end{lemma}
\begin{proof}
We choose $D_2$ and $s$ according to Lemma~\ref{lemma-isolate} for the function $h(s)=34(s-2)(s+3)+474$.
By Lemma~\ref{lemma-isolate}, there exists an integer $d_2\ge D_1$ such that $d_3=d_2+h(s)\le D_2$
and some $S\in\SS$ is $(d_2,d_3,s)$-isolated by some slice~$L$.  By Lemma~\ref{lem:cylsep}
applied to $(G,\{S\},C)$, $G$ contains a clean joint $H$ with $V(H)\subseteq L^\circ$.
Consequently, $H$ is vertex-disjoint from $C$ and at distance at least $d_2+1>D_1$ and at most $d_3-1\le D_2-1$ from $S$.
\end{proof}

\section{Colorings of quadrangulations of a cylinder}\label{sec-cyl}

In this section, we give a lemma on extending a precoloring of
boundaries of a quadrangulated cylinder.
This is a special case of a more general theory which we develop in the following
paper of the series~\cite{trfree6}.

Let $C$ be a cycle drawn in plane, let $v_1,v_2,\ldots,v_k$ be the vertices of $C$ listed in the clockwise order
of their appearance on $C$, and let
$\varphi:V(C)\to\{1,2,3\}$ be a $3$-coloring of $C$. We can view $\varphi$
as a mapping of $V(C)$ to the vertices of a triangle, and speak of the
{\em winding number of $\varphi$ on $C$}, defined as the number of indices
$i\in\{1,2,\ldots,k\}$ such that $\varphi(v_i)=1$ and $\varphi(v_{i+1})=2$
minus the number of indices $i$ such that $\varphi(v_i)=2$ and $\varphi(v_{i+1})=1$,
where $v_{k+1}$ means $v_1$.  We denote the winding number of $\varphi$ on $C$ by $W_\varphi(C)$.

Consider a plane graph $G$ and its $3$-coloring $\varphi$.  For a face $f$ of $G$ bounded by a cycle $C$,
we define the {\em winding number of $\varphi$ on $f$}, which is denoted by $w_\varphi(f)$, as $-W_\varphi(C)$ if $f$ is the outer face of $G$
and as $W_\varphi(C)$ otherwise.  The following two propositions are easy to prove.

\begin{proposition}
\mylabel{wsum0}
Let $G$ be a plane graph such that every face of $G$ is bounded by a cycle, and let
$\varphi:V(G)\to\{1,2,3\}$ be a $3$-coloring of $G$. Then the sum of the winding
numbers of all the faces of $G$ is zero.
\end{proposition}

\begin{proposition}
\mylabel{wC4}
The winding number of every $3$-coloring on a cycle of length four is zero.
\end{proposition}

Let $G$ be a cylindrical quadrangulation with boundary faces $f_1$ and $f_2$.
We say that the cylindrical quadrangulation is \emph{boundary-linked} if every cycle $K$ in $G$ separating $f_1$ from $f_2$
and not bounding either of these faces has length at least $\max(|f_1|,|f_2|)$,
and if $|K|=|f_i|=\max(|f_1|,|f_2|)$ for some $i\in\{1,2\}$, then $V(K)\cap V(f_{3-i})\neq\emptyset$.
The cylindrical quadrangulation is \emph{long} if the distance between $f_1$ and $f_2$ is at least $|f_1|+|f_2|$.

\begin{lemma}\label{lemma-nbr}
Let $G$ be a long boundary-linked cylindrical quadrangulation with boundary faces $f_1$ and $f_2$ and let $\psi$ be a $3$-coloring
of the boundary of $G$.
Suppose that $|f_1|\ge \max(5,|f_2|)$ and let $v_1v_2v_3$ be a subpath of the cycle bounding $f_1$, where $\psi(v_1)=\psi(v_3)$.
Then, there exists a long boundary-linked cylindrical quadrangulation $G'$ with boundary faces $f_1'$ and $f_2'$ such that
$|f'_1|=|f_1|-2$ and $|f'_2|=|f_2|$ together with a $3$-coloring $\psi'$ of the boundary of $G'$ such that
$w_{\psi'}(f'_1)=w_{\psi}(f_1)$, $w_{\psi'}(f'_2)=w_{\psi}(f_2)$, and if $\psi'$ extends to a $3$-coloring of $G'$,
then $\psi$ extends to a $3$-coloring of $G$.
\end{lemma}
\begin{proof}
Note that since $\max(|f_1|,|f_2|)\ge 5$ and $G$ is boundary-linked, it follows that $G$ contains no triangle other than possibly the cycle bounding $f_2$, and thus the neighbors of $v_2$ form an independent set in $G_2$.
Furthermore, $f_1$ is an induced cycle.  Let $G'$ be the cylindrical quadrangulation obtained from $G-v_2$
by contracting all neighbors of $v_2$ (including $v_1$ and $v_3$) to a single vertex $w$ and by suppressing the arising $2$-faces.
Let $f'_1$ and $f'_2$ be the faces of $G'$ corresponding to $f_1$ and $f_2$, respectively.  Clearly, $G'$ is long.

Let $\psi'$ be the coloring of the boundary of $G'$ such that $\psi'(w)=\psi(v_1)$ and $\psi'(z)=\psi(z)$ for all vertices $z\neq w$
in the boundary.
If $\psi'$ extends to a $3$-coloring $\varphi$ of $G'$, then we can turn $\varphi$ into a $3$-coloring of $G$ extending $\psi$ by setting $\varphi(z)=\psi(v_1)$ for every neighbor $z$ of $v_2$
and $\varphi(v_2)=\psi(v_2)$. 

Consider a cycle $K'$ separating $f'_1$ from $f'_2$ in $G'$ and not bounding either of these faces.  Let $K$ be the corresponding cycle in $G$ (equal to $K'$, or obtained from $K'$
by replacing $w$ by a neighbor of $v_2$, or obtained from $K'$ by replacing $w$ by a path $xv_2y$ for some neighbors $x$ and $y$ of $v_2$).

Let us first consider the case that $|f_1|>|f_2|$.  Note that $|f_1|$ and $|f_2|$ have the same parity, and thus $|f_1|\ge |f_2|+2$ and
$|f_1'|\ge |f_1|-2\ge |f_2|$.  Consequently, $|K'|\ge |K|-2\ge |f_1|-2=\max(|f'_1|,|f'_2|)$.  Furthermore, the equality only holds if $v_2\in V(K)$
and $|K|=|f_1|$.  Since $G$ is boundary-linked, the latter implies that $K$ also contains a vertex incident with $f_2$.  However,
this contradicts the assumption that $G$ is long.  Therefore, we have $|K'|>\max(|f'_1|,|f'_2|)$.

Next, we consider the case that $|f_1|=|f_2|$, and thus $\max(|f'_1|,|f'_2|)=|f_2|>|f'_1|$.  If $|K|=|f_2|$, then since $G$ is boundary-linked, it would follow that $K$ intersects both $f_1$ and $f_2$,
contrary to the assumption that $G$ is long.  Therefore, $|K|>|f_2|$, and by parity, $|K|\ge |f_2|+2$.  Consequently, $|K'|\ge |K|-2\ge |f_2|$.  The equality can only hold
when $K$ contains $v_2$, and thus $K'$ contains the vertex $w$ incident with $f'_1$.  We conclude that $G'$ is boundary-linked.
\end{proof}

\begin{lemma}\label{lemma-elim4}
Let $G$ be a long cylindrical quadrangulation with boundary faces $f_1$ and $f_2$ and let $\psi$ be a $3$-coloring
of the boundary of $G$.  If $|f_1|=|f_2|=4$, then $\psi$ extends to a $3$-coloring of $G$.
\end{lemma}
\begin{proof}
Let $v_1v_2v_3v_4$ be the cycle bounding $f_1$.  Since $\psi$ uses only three colors, we can without loss of generality assume
$\psi(v_1)=\psi(v_3)$.
Note that $G$ is bipartite, and thus the vertices at distance exactly three from $\{v_2,v_4\}$ form an independent set.
Let $G'$ be the quadrangulation of the plane obtained from $G$ by removing all vertices at distance at most two from $\{v_2,v_4\}$, identifying all vertices at
distance exactly three from $\{v_2,v_4\}$ to a single (non-boundary) vertex $w$ and by suppressing the arising $2$-faces.

Let $\psi'$ be a restriction of $\psi$ to the $4$-cycle bounding the face of $G'$ corresponding to $f_2$. 
By Lemma~\ref{lemma-extdisk}, $\psi'$ extends to a $3$-coloring $\varphi$ of $G'$.
We can extend $\varphi$ to a $3$-coloring of $G$ as follows. Give all vertices at distance exactly $1$ from $\{v_2,v_4\}$ the color $\psi(v_1)=\psi(v_3)$,
all vertices at distance exactly $3$ from $\{v_2,v_4\}$ the color $\varphi(w)$ and all vertices at distance exactly $2$ from $\{v_2,v_4\}$ an arbitrary color
different from $\psi(v_1)$ and $\varphi(w)$.  The resulting assignment is a $3$-coloring of $G$ extending $\psi$.
\end{proof}

Next, we aim to use the connection between colorings and nowhere-zero flows first noticed by Tutte~\cite{tutteflow}.
We only need the following implication from flows to colorings.
A \emph{nowhere-zero $\mathbb{Z}_3$-flow} in a graph $G$ is an orientation of $G$ such that the difference
between the indegree and the outdegree of each vertex is divisible by $3$.
Given an orientation $\vec{G}^\star$ of the dual $G^\star$ of a connected plane graph $G$
and a directed edge $e\in E(\vec{G}^\star)$,
we define $l(e)=u$ and $r(e)=v$, where $uv$ is the edge of $G$ crossing $e$ and $u$ is to
the left of $e$.
\begin{proposition}\label{prop-flows}
Let $G$ be a connected plane graph and let $G^\star$ be its dual.  If $\vec{G}^\star$ is a nowhere-zero $\mathbb{Z}_3$-flow,
then $G$ has a $3$-coloring $\varphi$ such that $\varphi(r(e))-\varphi(l(e))\equiv 1\pmod 3$ for every $e\in E(\vec{G}^\star)$.
\end{proposition}

We say that a $3$-coloring $\psi$ of a cycle $C=v_1\ldots v_k$ is \emph{rotating} if $3|k$,
$\psi(v_1)=\psi(v_4)=\ldots=\psi(v_{3k-2})$, $\psi(v_2)=\psi(v_5)=\ldots=\psi(v_{3k-1})$, and $\psi(v_3)=\psi(v_6)=\ldots=\psi(v_{3k})$.
Note that for any $3$-coloring $\psi$ of $C$, we have $W_\psi(C)\le |C|/3$, with equality if and only if $\psi$ is rotating.

\begin{lemma}\label{lemma-cyl}
Let $G$ be a long boundary-linked cylindrical quadrangulation with boundary faces $f_1$ and $f_2$
and let $\psi$ be a $3$-coloring of the boundary of $G$.
The coloring $\psi$ extends to a $3$-coloring of $G$ if and only if $w_\psi(f_1)+w_\psi(f_2)=0$.
\end{lemma}
\begin{proof}
If $\psi$ extends to a $3$-coloring of $G$, then $w_\psi(f_1)+w_\psi(f_2)=0$ by Propositions~\ref{wsum0} and \ref{wC4}.

Let us now show the converse implication.  We proceed by induction on $|f_1|+|f_2|$, and thus we can assume that the claim
holds for all graphs whose boundary has less than $|f_1|+|f_2|$ vertices.
By symmetry, we can assume that $|f_1|\ge |f_2|$.

If $|f_1|=4$, then since $|f_1|$ and $|f_2|$ have
the same parity, we have $|f_2|=4$, and $\psi$ extends to a $3$-coloring of $G$ by Lemma~\ref{lemma-elim4}.
Thus, assume $|f_1|\ge 5$.

If the cycle bounding $f_1$ contains a path $v_1v_2v_3$ with $\psi(v_1)=\psi(v_3)$, then $\psi$ extends to a $3$-coloring of $G$ by
Lemma~\ref{lemma-nbr} and the induction hypothesis.
Therefore, we can assume that the boundary cycle of $f_1$ contains no such path, and thus $\psi$ is rotating on this cycle.
It follows that $|f_1|$ is a multiple of $3$ and $|w_\psi(f_1)|=|f_1|/3$.  Since $w_\psi(f_1)+w_\psi(f_2)=0$, we have $|w_\psi(f_2)|=|f_1|/3$,
and since $|f_2|\le |f_1|$, we conclude that $\psi$ is also rotating on the boundary of $f_2$ and $|f_2|=|f_1|$.
Since $G$ is long and boundary-linked, every cycle in $G$ that separates $f_1$ from $f_2$ and does not bound either of the
faces has length at least $|f_1|+2$.

Let $G^\star$ be the dual of $G$.
Let $K_i$ be the edge-cut in $G$ consisting of the edges incident with $V(f_i)$ that do not belong to $E(f_i)$.  Note that the dual $K^\star_i$ of $K_i$ is
a cycle in $G^\star$.  Let $H=G^\star-(E(K^\star_1)\cup E(K^\star_2))$.  Let $f^\star_1$ and $f^\star_2$ be the vertices of the dual corresponding to $f_1$ and $f_2$, respectively.
Suppose that $H$ contains an edge-cut of size less than $|f_1|$ separating
$f^\star_1$ from $f^\star_2$, and thus $G^\star$ contains an edge cut $K^\star$ separating
$f^\star_1$ from $f^\star_2$ with less than $|f_1|$ edges belonging to $E(K^\star_1)\cup E(K^\star_2)$.
Choose $K^\star$ as a minimal edge-cut with this property; then the dual $K$ to $K^\star$ is a cycle in $G$
separating $f_1$ from $f_2$ such that $|E(K)\setminus (E(K_1)\cup E(K_2))|<|f_1|$.  In particular, this implies $K$ bounds neither $f_1$ nor $f_2$.
Since $G$ is long, $K$ does not intersect both $K_1$ and $K_2$.  As we observed before, $|K|\ge |f_1|+2$,
and thus we can by symmetry assume that $K$ intersects $K_1$ in at least three edges.
Let us choose such a cycle $K$ that shares as many edges with the cycle bounding $f_1$ as possible.
Let $P$ be a subpath of $K$ with both endpoints incident with $f_1$, but no other vertex or edge incident with $f_1$.
Let $Q_1$ and $Q_2$ be the two subpaths of the cycle bounding $f_1$ joining the endpoints of $P$ labelled so that $P\cup Q_2$
is a cycle separating $f_1$ from $f_2$.  Consider the cycle $K'=(K-P)\cup Q_1$.  Since $K$ intersects $K_1$ in at least three edges,
$K'$ is not the cycle bounding $f_1$.  Since $K'$ shares more edges with the cycle bounding $f_1$ than $K$,
the choice of $K$ implies that
\begin{align*}
|E(K')\setminus (E(K_1)\cup E(K_2))|\ge|f_1|&>|E(K)\setminus (E(K_1)\cup E(K_2))|\text{, and thus}\\
|E(Q_1)\setminus (E(K_1)\cup E(K_2))|&>|E(P)\setminus (E(K_1)\cup E(K_2))|.
\end{align*}
Since $|E(Q_1) \cap (E(K_1)\cup E(K_2))|=0$ and
$|E(P) \cap (E(K_1)\cup E(K_2))|=2$, we conclude that $|Q_1|>|P|-2$.  However, then
the cycle $P\cup Q_2$ has length less than $|f_1|+2$, contradicting the assumption that $G$ is boundary-linked.

Therefore, $H$ does not contain any edge-cut of size less than $|f_1|$ separating
$f^\star_1$ from $f^\star_2$, and by Menger's theorem, $H$ contains pairwise edge-disjoint paths $P_1$, \ldots, $P_{|f_1|}$ joining $f^\star_1$ with $f^\star_2$.
Note that all vertices of $H'=H-E(P_1\cup P_2\cup\ldots\cup P_{|f_1|})$ have even degree, and thus $H'$ is a union of pairwise edge-disjoint cycles $C_1$, \ldots, $C_m$.
For $1\le i\le m$, direct the edges of $C_i$ so that all vertices of $C_i$ have outdegree $1$.  For $1\le i\le |f_1|$, direct the edges of $P_i$ so that all its vertices except for $f_1^\star$ have outdegree $1$.
This gives an orientation $\vec{H}$ of $H$ such that the indegree of every vertex of $V(H)\setminus \{f^\star_1,f^\star_2\}$ equals its outdegree, $f^\star_1$ has outdegree $0$ and $f^\star_2$ has indegree $0$.
Let $\vec{G}^\star_1$ be the orientation of $G^\star$ obtained from $\vec{H}$ by orienting all edges of $K^\star_1$ and $K^\star_2$ in the clockwise direction along the cycles.
Let $\vec{G}^\star_2$ be the orientation of $G^\star$ obtained from $\vec{G}^\star_1$ by reversing the orientation of the edges of $K^\star_1$, and let
$\vec{G}^\star_3$ be the orientation of $G^\star$ obtained from $\vec{G}^\star_2$ by reversing the orientation of the edges of $K^\star_2$.

Since $|f_1|=|f_2|$ is a multiple of $3$, it follows that the orientations $\vec{G}^\star_1$, $\vec{G}^\star_2$ and $\vec{G}^\star_3$ define nowhere-zero $\mathbb{Z}_3$-flows in $G^\star$.
Let $\varphi_1$, $\varphi_2$ and $\varphi_3$ be the corresponding $3$-colorings of $G$ arising from Proposition~\ref{prop-flows}.
Since $f_1^\star$ has outdegree $0$ in all three orientations, these $3$-colorings are rotating on the boundary of $f_1$,
and thus we can permute the colors so that the restrictions of $\varphi_1$, $\varphi_2$, and $\varphi_3$ to the cycle bounding $f_1$ match $\psi$.
Similarly, for $i\in\{1,2,3\}$, the coloring $\varphi_i$ is rotating on the boundary of $f_2$. Propositions~\ref{wsum0} and \ref{wC4}
imply $w_{\varphi_i}(f_1)+w_{\varphi_i}(f_2)=0$, and since $w_\psi(f_1)+w_\psi(f_2)=0$ and $w_\psi(f_1)=w_{\varphi_i}(f_1)$,
we conclude $w_{\varphi_i}(f_2)=w_\psi(f_2)$.  Consequently, the restrictions of $\varphi_1$, $\varphi_2$ and $\varphi_3$ to the boundary of $f_2$
differ from $\psi$ only by a cyclic permutation of colors.
Observe that the colors $\varphi_1(v)$, $\varphi_2(v)$ and $\varphi_3(v)$ are pairwise distinct for every $v\in V(f_2)$, since the reversals of
the orientations of $K^\star_1$ and $K^\star_2$ cyclically permute the colors on the boundary of $f_2$.
Consequently, one of these colorings matches $\psi$ on the boundary of $f_2$, and thus there exists $i\in\{1,2,3\}$ such
that $\varphi_i$ is a $3$-coloring of $G$ extending $\psi$.
\end{proof}

The inspection of the proofs of Lemmas~\ref{lemma-nbr}, \ref{lemma-elim4}, and \ref{lemma-cyl} shows that they are constructive
and can be implemented as linear-time algorithms to find the described $3$-colorings (Lemma~\ref{lemma-extdisk} is only used in the proof of Lemma~\ref{lemma-elim4}
to extend the precoloring of a $4$-cycle, and a linear-time algorithm for this special case appears in~\cite{DvoKawTho}).  Hence, we obtain the following
corollary which we use in the next paper of the series~\cite{trfree6}.

\begin{corollary}\label{cor-cylqalg}
For all positive integers $d_1$ and $d_2$, there exists a linear-time algorithm as follows.
Let $G$ be a cylindrical quadrangulation with boundary faces $f_1$ and $f_2$
and let $\psi$ be a $3$-coloring of the boundary of $G$ such that $w_\psi(f_1)+w_\psi(f_2)=0$.
Suppose that $|f_1|=d_1$, $|f_2|=d_2$, 
every cycle in $G$ separating $f_1$ from $f_2$ and not bounding either of these faces has length greater than $\max(d_1,d_2)$,
and the distance between $f_1$ and $f_2$ is at least $d_1+d_2$.
Then the algorithm returns a $3$-coloring of $G$ that extends $\psi$.
\end{corollary}

We also need another result similar to Lemma~\ref{lemma-cyl}.

\begin{corollary}\label{cor-cyl}
Let $G$ be a joint with boundary faces $f_1$ and $f_2$ and let $\psi$ be a $3$-coloring of the boundary of $G$ such that $w_\psi(f_1)+w_\psi(f_2)=0$.
If $|w_\psi(f_1)|<|f_1|/3$, then $\psi$ extends to a $3$-coloring of $G$.
\end{corollary}
\begin{proof}
Since $|w_\psi(f_1)|<|f_1|/3$, we have $|f_1|\neq 3$.  If $|f_1|=4$, then $\psi$ extends to a $3$-coloring of $G$ by Lemma~\ref{lemma-elim4}.
Therefore, assume $|f_1|\ge 5$.
Since $|w_\psi(f_1)|<|f_1|/3$ and $|w_\psi(f_2)|<|f_2|/3$, the coloring $\psi$ is not rotating on the boundaries of $f_1$ and $f_2$, and thus
there exist paths $u_1u_2u_3$ and $v_1v_2v_3$ in the cycles bounding $f_1$ and $f_2$,
respectively, such that $\psi(u_1)=\psi(u_3)$ and $\psi(v_1)=\psi(v_3)$.  Let $G'$ be the cylindrical quadrangulation obtained from $G-u_2-v_2$ by
identifying all neighbors of $u_2$ to a single vertex $w_1$ and all neighbors of $v_2$ to a single vertex $w_2$.
Let $\psi'$ be the coloring of the boundary of $G'$ such that $\psi'(w_1)=\psi(u_1)$, $\psi'(w_2)=\psi(v_1)$ and $\psi'(z)=\psi(z)$ for any other
boundary vertex of $G'$.  Clearly, it suffices to show that $\psi'$ extends to a $3$-coloring of $G'$.

Let $f'_1$ and $f'_2$ be the boundary faces of $G'$ corresponding to $f_1$ and $f_2$, respectively.  Note that every cycle in $G'$ separating $f'_1$
from $f'_2$ has length at least $|f'_1|$, and each such cycle of length $|f'_1|$ contains either $w_1$ or $w_2$.  We can assume that $G'$ is drawn
so that $f'_1$ is its outer face.  Let $A$ be a subset of
the plane homeomorphic to the closed annulus such that the boundary of $A$ is formed by cycles in $G'$ of length $|f'_1|$ separating $f'_1$ from $f'_2$,
one of them containing $w_1$, the other one containing $w_2$, such that no other cycle separating $f'_1$ from $f'_2$ is contained in $A$.  Let $G_0$ be the subgraph of $G'$ drawn in $A$.
Removing $A$ splits the plane into two connected open sets $B_1$ and $B_2$, where $f'_1\subset B_1$.
For $i\in \{1,2\}$, let $G_i$ be the subgraph of $G'$ drawn in the closure of $B_i$.  Note that $G_0$ is a long boundary-linked cylindrical quadrangulation.
By Lemma~\ref{lemma-extdisk}, $\psi'$ extends to a $3$-coloring of $G_1\cup G_2$, and by Lemma~\ref{lemma-cyl}, the resulting coloring of the boundary
of $G_0$ extends to a $3$-coloring of $G_0$.  This gives a $3$-coloring of $G'$ extending $\psi'$.
\end{proof}

To use the results of this section, we need means to constrain the winding number of a coloring on a boundary of a face.
We achieve this by filling the face by a carefully chosen cylindrical quadrangulation.
An \emph{$s$-cap} is a cylindrical quadrangulation $G$ with boundary faces $f_1$ and $f_2$, such that $G$ does not contain triangles and separating $4$-cycles, $|f_1|=s$,
$|f_2|=4+(s \bmod 2)$ and for every $u,v\in V(f_1)$, the distance between $u$ and $v$ in $G$ is the same as their distance in the cycle bounding $f_1$.
We call $f_2$ the \emph{special face} of the $s$-cap.

\begin{lemma}\label{lemma-excap}
For every $s\ge 4$, there exists an $s$-cap $G$ that has fewer vertices than every joint with boundary faces of length $s$.
\end{lemma}
\begin{proof}
Let $G$ be an $s$-cap obtained from the $s\times s$ cylindrical quadrangulation by adding chords to one of its boundary faces.
We have $|V(G)|=s^2$.

Consider any joint $H$ with boundary faces $f_1$ and $f_2$ of length $s$.  For $1\le i\le 4s-1$, let $V_i$ denote the set of vertices
of $H$ at distance exactly $i$ from $f_1$.  Observe that since all faces of $H$ other than $f_1$ and $f_2$ have length $4$,
$H[V_i\cup V_{i+1}]$ contains a cycle separating $f_1$ from $f_2$ for $1\le i\le 4s-2$, and thus $|V_i|+|V_{i+1}|\ge s$.
Therefore, $|V(H)|\ge |f_1|+|f_2|+(2s-1)s=(2s+1)s>|V(G)|$.
\end{proof}

\section{$3$-coloring with distant anomalies}\label{sec-havel}

An \emph{anomaly} is a triple $T=(H_T,B_T,\Phi_T)$, where $H_T$ is a connected plane graph,
$B_T\subseteq V(H_T)$ and $\Phi_T$ is a set of $3$-colorings of $H_T$ such that
for every $\psi\in \Phi_T$, there exist distinct colors $a$ and $b$ such that the $3$-coloring obtained from $\psi$ by swapping the colors $a$ and $b$ also belongs to $\Phi_T$.
An anomaly $T$ \emph{appears} in a plane graph $G$ if $H_T$ is an induced subgraph of $G$ (where the plane embedding of $H_T$ is induced by the embedding of $G$)
and every $v\in B_T$ satisfies $\deg_G(v)=\deg_{H_T}(v)$.
Given a $3$-coloring $\varphi$ of a plane graph $G$ and an anomaly $T$ appearing in $G$, we say that $\varphi$ is
\emph{compatible with $T$} if $\varphi\restriction V(H_T)\in \Phi_T$.

An anomaly $T$ is \emph{locally extendable} if the following holds for every plane graph $G$: if $T$ appears in $G$
and all triangles in $G$ are contained in $H_T$, then there exists a $3$-coloring of $G$ compatible with $T$.
For an integer $r\ge 0$, an anomaly $T$ is \emph{strongly locally extendable with margin $r$} if for every plane graph $G$ in that $T$ appears so that all triangles of $G$ are contained in $H_T$,
and for every $4$-face $f$ of $G$ at distance at least $r$ from $H_T$, every $3$-coloring $\psi$ of the boundary of $f$ extends to a $3$-coloring of $G$ compatible with $T$.

The following anomalies are of interest for Theorems~\ref{havel} and \ref{thm-nosep4sp}.
Recall that the pattern of a $3$-coloring $\psi$ is the set $\{\psi^{-1}(1),\psi^{-1}(2),\psi^{-1}(3)\}$.
\begin{itemize}
\item A single precolored vertex ($H_T$ is a single vertex, $B_T$ is empty and $\Phi_T$ consists of a coloring assigning to the vertex of $H_T$ the prescribed color).
This anomaly is locally extendable by Gr\"otzsch' theorem.  It is also strongly locally extendable with some margin,
as we hypothesized in Conjecture~\ref{conj-fourext} and was later proved in~\cite{cylgen-part2}.
\item A cycle of length at most $5$ with a prescribed pattern of coloring ($H_T$ is a $(\le\!5)$-cycle, $B_T$ is empty and $\Phi_T$ consists of all $3$-colorings of $H_T$
with the prescribed pattern).  This anomaly is locally extendable by Lemma~\ref{lem:aksionov}.  Furthermore, the same lemma implies that if the cycle has length $3$, then
the anomaly is strongly locally extendable with margin $0$.
\item A vertex of degree at most $4$ with neighborhood precolored by one color
($H_T$ is a star with at most $4$ rays, $B_T$ contains the center of the star and $\Phi_T$ consists of all $3$-colorings of $H_T$ which assign the prescribed color to the rays).
This anomaly is locally extendable by the results of Gimbel and Thomassen~\cite{gimbel} for degree at most $3$ and Dvo\v{r}\'ak and Lidick\'y~\cite{col8cyc}
for degree $4$ (given a vertex $v$ of degree $k\le 4$ with precolored neighborhood, split $v$ into $k$ vertices of degree two colored arbitrarily and
extend the coloring of the resulting $2k$-cycle).
\end{itemize}

Thus, both Theorem~\ref{havel} and Theorem~\ref{thm-nosep4sp} are implied by the following general statement (which also shows that Conjecture~\ref{conj-fourext} implies
Conjecture~\ref{conj-farsv}),
by letting $C$ be the null graph, $p=5$ and $r=0$.

\begin{theorem}\label{thm-supernosep4}
For all integers $p\ge 1$ and $r\ge 0$, there exist constants $0<d_0<d_1$ with the following property.
Let $G$ be a plane graph and let $\TT=\{T_i:1\le i\le n\}$ be a set of locally extendable anomalies appearing in $G$, such that $|V(H_{T_i})|\le p$ for $1\le i\le n$.
Let $C$ be either the null graph or a facial cycle of $G$ of length at most five, at distance at least $2d_0$ from $H_T$ for each $T\in \TT$.
Suppose that
\begin{itemize}
\item for $1\le i < j\le n$, the distance between $H_{T_i}$ and $H_{T_j}$ in $G$ is at least $2d_1$,
\item every triangle in $G$ distinct from $C$ is contained in $H_T$ for some $T\in \TT$, and
\item if a separating $4$-cycle $K$ is at distance less than $2d_0$ from $H_T$ for some $T\in \TT$,
then either $K$ is contained in $H_T$, or $T$ is strongly locally extendable with margin $r$.
\end{itemize}
Then, every $3$-coloring of $C$ extends to a $3$-coloring of $G$ compatible with all elements of $\TT$.
\end{theorem}
\begin{proof}
For the function $f_{\ref{lemma-struct}}:\mathbb{N}^2\to\mathbb{N}$ from Lemma~\ref{lemma-struct}, let
$d_0=\max(r, f_{\ref{lemma-struct}}(r+4,p))+1$ and $d_1=\max(2d_0,f_{\ref{lemma-struct}}(2d_0+3,p))$.
We will prove by induction on $|V(G)|$ that $d_0$ and $d_1$ satisfy the conclusion of the theorem.

Let $G$ be as stated, let $\psi$ be a $3$-coloring of $C$, and assume
for a contradiction that $\psi$ does not extend to a $3$-coloring of $G$ compatible with all elements of $\TT$.
Let $\SS=\{V(H_T):T\in \TT\}$, $Z_0=\bigcup_{S\in \SS} G[S]$ and $Z=C\cup Z_0$.
For a set $X\subseteq V(G)$, let $\TT[X]=\{T\in\TT : V(H_T)\subseteq X\}$.
Note that $G$ is connected, as otherwise we can color each component of $G$ separately by the induction hypothesis.
Without loss of generality, we can assume that if $C$ is not null, then it bounds the outer face of $G$.
Hence, $(G,\SS,C)$ is a $2d_1$-distant $p$-small internally triangle-free scene.
Note also that if $C$ is not null then $C$ is an induced cycle, since otherwise a triangle containing a chord of $C$ would be contained
in $H_T$ for some $T\in\TT$ and the distance between $H_T$ and $C$ would be zero, contradicting the assumptions.

\claim{cl-nojoint}{Suppose $H$ is a clean joint in $G$ vertex-disjoint from $Z$, with boundary faces $f_1$ and $f_2$ labelled so that the face of $G$ bounded by $C$ (if any) is contained in
$f_1$. For $i\in\{1,2\}$, let $G'_i$ be the subgraph of $G$ drawn in the closure of $f_i$.  Then $|\TT[V(G'_2)]|\ge 2$ and $H$ is at distance less than $2d_0$ from $H_T$ in $G$
for some $T\in \TT[V(G'_2)]$.}
\begin{subproof}
Suppose for a contradiction that either $|\TT[V(G'_2)]|\le 1$ or $H$ is at distance at least $2d_0$ from every subgraph $H_T$ with $T\in \TT[V(G'_2)]$.

For $i\in\{1,2\}$, let $H_i$ be an $|f_i|$-cap with its non-special boundary
cycle equal to the boundary of $f_i$, but otherwise disjoint from $G'_i$, such that $|V(H_i)|<|V(H)|$, which exists
by Lemma~\ref{lemma-excap}. Let $h_i$ be the special face of $H_i$. Let $G_i=G'_i+H_i$.  Note that the distance between any two elements of $\SS\cup \{C\}$
in $G_i$ is the same as the distance between them in $G'_i$, which is greater or equal to their distance in $G$.
By the induction hypothesis, $\psi$ extends to a $3$-coloring $\varphi_1$ of $G_1$ compatible with all the elements of $\TT[V(G'_1)]$.
Consider the restriction of $\varphi_1$ to $H_1$.  Propositions~\ref{wsum0} and \ref{wC4} imply that $w_{\varphi_1}(f_1)+w_{\varphi_1}(h_1)=0$.
Furthermore, since $h_1$ has length at most $5$, we have $w_{\varphi_1}(h_1)=0$ if $|h_1|=4$ ($f_1$ has even length) and
$|w_{\varphi_1}(h_1)|=1$ if $|h_1|=5$ ($f_1$ has odd length).

We now obtain a $3$-coloring $\varphi_2$ of $G_2$ compatible with all the elements of $\TT[V(G'_2)]$ such that $w_{\varphi_2}(h_2)=w_{\varphi_1}(f_1)$.
Let $C_2$ be the cycle bounding $h_2$.
\begin{itemize}
\item Suppose $\TT[V(G'_2)]=\emptyset$.  Since $h_2$, $h_1$, and $f_1$ have the same parity and $|w_{\varphi_1}(f_1)|\le 1$, there exists
a $3$-coloring $\psi_2$ of $C_2$ such that $w_{\psi_2}(h_2)=w_{\varphi_1}(f_1)$.  Since $G_2$ is planar and triangle-free,
$\psi_2$ extends to a $3$-coloring $\varphi_2$ of $G_2$ by Lemma~\ref{lem:aksionov}.
\item Suppose $|\TT[V(G'_2)]|=1$.  Then there exists a $3$-coloring $\varphi'_2$ of $G_2$ compatible with $T$ by the local extendability of $T$.
Let $a$ and $b$ be distinct colors such that the $3$-coloring $\varphi''_2$ obtained from $\varphi'_2$ by swapping the colors $a$ and $b$ is also compatible
with $T$.  Note that $w_{\varphi'_2}(h_2)=-w_{\varphi''_2}(h_2)$, $|w_{\varphi'_2}(h_2)|\le 1$ and $w_{\varphi'_2}(h_2)$ and $w_{\varphi_1}(f_1)$ have the same
parity, and thus we can choose $\varphi_2$ as one of $\varphi'_2$ and $\varphi''_2$.
\item Suppose $|\TT[V(G'_2)]|\ge 2$, and thus $H$ is at distance at least $2d_0$ from every subgraph $H_T$ with $T\in \TT[V(G'_2)]$.
Choose $\psi_2$ be an arbitrary $3$-coloring of $C_2$ such that $w_{\psi_2}(h_2)=w_{\varphi_1}(f_1)$.
The distance from $C_2$ to any subgraph $H_T$ with $T\in \TT[V(G'_2)]$ is also at least $2d_0$, and thus by the induction hypothesis, $\psi_2$
extends to a $3$-coloring $\varphi_2$ of $G_2$ compatible with all elements of $\TT[V(G'_2)]$.
\end{itemize}
By Propositions~\ref{wsum0} and \ref{wC4} for $H_2$, we have $w_{\varphi_2}(f_2)=-w_{\varphi_2}(h_2)=-w_{\varphi_1}(f_1)$.
By Corollary~\ref{cor-cyl}, the restriction of $\varphi_1\cup \varphi_2$ to the boundary cycles
of $f_1$ and $f_2$ extends to a $3$-coloring $\varphi_3$ of $H$.  Consequently, the restriction of $\varphi_1$ to $G'_1$,
the restriction of $\varphi_2$ to $G'_2$, and $\varphi_3$ together give a $3$-coloring of $G$ extending $\psi$ and compatible with all the elements of $\TT$.
This is a contradiction.
\end{subproof}

We may assume, by taking a subgraph of $G$, that $\psi$ extends to a $3$-coloring compatible with all elements of $\TT$
for every proper subgraph of $G$ that includes $Z$.  Using the fact that $G$ is connected, we have $G\neq Z$, as otherwise either
$\TT=\emptyset$, $G=C$, and the claim is trivial, or $C$ is the null graph and $|\TT|=1$ and the claim follows by the local extendability
of the anomaly in $\TT$.  Consequently, $G$ is $Z$-critical.

If $K$ is a separating $(\le\!5)$-cycle and $\Delta_K$ is the open disk in the plane bounded by $K$,
then at least one vertex or edge of $Z$ is drawn in $\Delta_K$, since $G$ is $Z$-critical and every $3$-coloring of a $(\le\!5)$-cycle
extends to a $3$-coloring of a triangle-free planar graph by Lemma~\ref{lem:aksionov}.
We claim that 
\claim{cl-nosep}{if $K$ is a separating cycle of length at most five in $G$, then
                 $K$ is at distance less than $2d_0$ from $Z_0$.  Furthermore, if $|K|\le 4$ and $K$ is $S$-private for some $S\in \SS$,
		 then the distance between $K$ and $S$ is less than $r$.}
\begin{subproof}
Without loss of generality, we can assume that $K$ does not have a chord $e$ drawn in $\Delta_K$; otherwise, $e$ is contained in a triangle,
and thus $K$ intersects $Z_0$, and moreover, if $|K|=4$ and $K$ is $S$-private, then one of the triangles in $K+e$ is $S$-private and we can consider it instead of $K$.

Suppose that for some anomaly $T\in\TT$, $H_T$ intersects $\Delta_K$ but is not contained in $\Delta_K$.
Since $K$ does not have a chord drawn in $\Delta_K$, a vertex of $H_T$ is drawn in $\Delta_K$, and thus if $K$ is $S$-private, then
$S=V(H_T)$.  Since $H_T$ is not contained in $\Delta_K$, it follows that $K$ is at distance $0$ from $H_T$, and the claim follows.

Let $G_1$ be the subgraph of $G$ drawn in the complement of $\Delta_K$ and $G_2$ the subgraph drawn in the closure of $\Delta_K$.
By the previous paragraph, we can assume the sets $\TT_1=\TT[V(G_1)]$ and $\TT_2=\TT[V(G_2)\setminus V(K)]$ partition $\TT$.
By the induction hypothesis, $G_1$ has a $3$-coloring $\varphi_1$ extending $\psi$ and compatible with all elements of $\TT_2$.
Since $\psi$ does not extend to a $3$-coloring of $G$ compatible with all elements of $\TT$, it follows the restriction of $\varphi_1$ to
$K$ does not extend to a $3$-coloring of $G_2$ compatible with all elements of $\TT_2$.  By the induction hypothesis, we conclude
that $K$ is at distance less than $2d_0$ from $H_T$ for some element $T\in\TT_1$.

Furthermore, if $K$ is $S$-private for some $S\in\SS$, then $\TT_1=\{T\}$ and $S=V(H_T)$.  If $K$ is a triangle, then since
$K$ is at distance less then $2d_0$ from $H_T$, the assumptions of this lemma imply $K\subseteq H_T$.  If $K$ is a $4$-cycle
not contained in $H_T$, then the assumptions of this lemma imply $H_T$ is strongly locally extendable with margin $r$,
and thus the distance between $K$ and $S$ is at most $r$ since the restriction of $\varphi_1$ to
$K$ does not extend to a $3$-coloring of $G_2$ compatible with~$T$.
\end{subproof}

In particular, the scene $(G,\SS,C)$ contains no $r$-distant private $4$-cycles.  We now consider $4$-faces of $G$.

\claim{cl-tight}{Let $f$ be a $4$-face of $G$ at distance at least $2d_0+3$ from $Z_0$.
     If $f$ is not bounded by $C$, then $f$ is $S$-tight for a unique set $S\in\SS$ at distance at most $d_1-1$ from~$f$.}
\begin{subproof}
Let the vertices of $f$ be numbered $u_1,u_2,u_3,u_4$ in order.
By (\ref{cl-nosep}), no vertex of $f$ is contained in a separating $4$-cycle.
Since additionally $C$ is an induced cycle if it is not null, the intersection of the boundary of $f$ with $C$ is a path of length
at most two.

If the intersection contains three vertices, say $u_1$, $u_2$ and $u_3$, then note that $u_2$ has degree two.
Consider the graph $G-u_2$ and color $u_4$ by $\psi(u_2)$.  By the induction hypothesis,
this coloring extends to a $3$-coloring of $G-u_2$ compatible with all elements of $\TT$, which also gives a $3$-coloring of $G$
extending $\psi$ and compatible with all elements of $\TT$, a contradiction.

Therefore, we can assume that $u_3,u_4\not\in V(C)$.  Note that $u_1u_2u_3$ and $u_1u_4u_3$
are the only paths of length at most three joining $u_1$ with $u_3$, as otherwise, since $f$ is at distance at least $2d_0+3$ from $Z_0$,
$G$ would contain a separating $(\le\!5)$-cycle contradicting (\ref{cl-nosep}).
Let $G_{13}$ be the graph obtained from $G$ by identifying $u_1$ and $u_3$ and suppressing
parallel edges, and observe that $G_{13}$ contains no new triangles.
Furthermore, $C$ as well as every new separating $4$-cycle in $G_{13}$ is at distance at least $2d_0$ from $Z_0$.
Let $G_{24}$ be defined analogously.

If $G_{13}$ or $G_{24}$ satisfies the assumptions of Theorem~\ref{thm-supernosep4}, then it has a $3$-coloring extending $\psi$
and compatible with all elements of $\TT$ by induction, which would give such a $3$-coloring of $G$.
Otherwise, both $G_{13}$ and $G_{24}$ contain a pair of anomalies at distance at most $2d_1-1$ from each other,
and thus $f$ is $S$-tight for a unique $S\in\SS$ at distance at most $d_1-1$ from $f$ by Lemma~\ref{lem:distcrit}.
\end{subproof}

Therefore, the scene $(G,\SS,C)$ is $(2d_0+3,d_1)$-tight.  If $|\SS|\ge 2$, then the choice of $d_1$ and Lemma~\ref{lemma-struct}
implies $G$ contains a clean joint vertex-disjoint from $C$ whose vertices are at distance at least $2d_0+3$ and at most $d_1-1$
from some element $S\in\SS$.  By (\ref{cl-nojoint}), $H$ is at distance less than $2d_0$ from some element $S'\in\SS$, necessarily distinct from $S$.
But then the distance between $S$ and $S'$ is less than $d_1+2d_0-1\le 2d_1$, contradicting the assumptions of this lemma.

Therefore, $|\SS|\le 1$.  If $\SS=\emptyset$, then $\psi$ extends to a $3$-coloring of $G$ by Lemma~\ref{lem:aksionov}.  Therefore, we can assume that $|\SS|=1$; let $\SS=\{S\}$ and $\TT=\{T\}$.
If $C$ is the null graph, then $G$ has a $3$-coloring compatible with $T$, since $T$ is locally extendable.
Hence, suppose that $C$ is a $(\le\!5)$-cycle.
By (\ref{cl-nosep}) and the assumptions of this theorem, if $T$ is not strongly locally extendable with margin $r$,
then all separating $4$-cycles of $G$ are contained in $H_T$.

\claim{cl-tight2}{Let $f$ be a $4$-face of $G$ at distance at least $r+4$ and at most $d_0-1$ from $S$.
     If $f$ is not $S$-tight, then $f$ is attached to a $(\le\!6)$-cycle separating $S$ from $C$.}
\begin{subproof}
Let the vertices of $f$ be numbered $u_1,u_2,u_3,u_4$ in order.
For $i\in\{1,2\}$, let $G_{i(i+2)}$ the graph obtained from $G$ by identifying $u_i$ with $u_{i+2}$ to a new vertex $z_i$ and suppressing
parallel edges.  If the distance between $S$ and $C$ in both $G_{13}$ and $G_{24}$ is less than $2d_0$, then
Lemma~\ref{lem:distcrit} applied to $(G,\{S,C\},\varnothing)$ implies $f$ is $S$-tight.
Hence, we can assume that the distance between $S$ and $C$ in $G_{13}$ is at least $2d_0$.

Suppose there exists a triangle in $G_{13}$ not contained in $H_T$, which was necessarily created by identification of $u_1$ with $u_3$.
Then $G$ contains a $5$-cycle $K=u_1u_2u_3xy$.  Since $G$ is $Z$-critical, $u_2$ has degree at least three, and thus $K$ does not bound a face.
Lemma~\ref{lem:aksionov} implies that $K$ separates $S$ from $C$, and thus the conclusion of the claim holds since $f$ is attached to $K$.
Therefore, we can assume every triangle in $G_{13}$ is contained in $H_T$.

Since $\psi$ does not extend to a $3$-coloring of $G$ compatible with $T$, $\psi$ also does not
extend to a $3$-coloring of $G_{13}$ compatible with $T$.  Let $G'_{13}$ be a minimal subgraph of $G_{13}$
containing $C$ and $H_T$ such that $\psi$ does not extend to a $3$-coloring of $G'_{13}$ compatible with $T$.
It follows that the induction hypothesis cannot apply to $G'_{13}$, and thus $T$ is not strongly locally extendable with margin $r$
and there exists a separating $4$-cycle $K'$ in $G'_{13}$ not contained in $H_T$, which was necessarily created by the identification of $u_1$ with $u_3$.
The minimality of $G'_{13}$ and Lemma~\ref{lem:aksionov} imply that $K'$ separates $S$ from $C$.  Let $K$ be the cycle in $G$ obtained from $K'$
by replacing $z_1$ by the path $u_1u_2u_3$.  Then $f$ is attached to the $6$-cycle $K$ separating $S$ from $C$.
\end{subproof}

Therefore, the scene $(G,\SS,C)$ is $(r+4,d_0-1)$-tight.  Since the distance between $S$ and $C$ is at least $2d_0>d_0-2$,
Lemma~\ref{lemma-struct} and the choice of $d_0$ implies $H$ contains a clean joint vertex-disjoint from $Z$.
Since $|\TT|=1$, this contradicts (\ref{cl-nojoint}) and finishes the proof.
\end{proof}


\begin{thebibliography}{10}

\bibitem{aksenov}
{\sc Aksionov, V.~A.}
\newblock On continuation of $3$-colouring of planar graphs.
\newblock {\em Diskret. Anal. Novosibirsk 26\/} (1974), 3--19.
\newblock In Russian.

\bibitem{aksmel}
{\sc Aksionov, V.~A., and Mel'nikov, L.~S.}
\newblock Some counterexamples associated with the {T}hree {C}olor {P}roblem.
\newblock {\em J. Combin. Theory, Ser.~B 28\/} (1980), 1--9.

\bibitem{Alb98}
{\sc Albertson, M.~O.}
\newblock You can't paint yourself into a corner.
\newblock {\em J. Combin. Theory, Ser.~B 73}, 2 (1998), 189--194.

\bibitem{BorIrred}
{\sc Borodin, O.~V.}
\newblock Irreducible graphs in the {G}r\"unbaum-havel 3-colour problem.
\newblock {\em Discrete Math. 159\/} (1996), 247--249.

\bibitem{bornew}
{\sc Borodin, O.~V.}
\newblock A new proof of {G}r{\"u}nbaum's 3 color theorem.
\newblock {\em Discrete Math. 169\/} (1997), 177--183.

\bibitem{4c4t}
{\sc Borodin, O.~V., Dvo\v{r}\'ak, Z., Kostochka, A., Lidick\'y, B., and
  Yancey, M.}
\newblock Planar 4-critical graphs with four triangles.
\newblock {\em European J. Combin. 41\/} (2014), 138--151.

\bibitem{fartr2}
{\sc Borodin, O.~V., Glebov, A.~N., and Jensen, T.~R.}
\newblock A step towards the strong version of {H}avel's three color
  conjecture.
\newblock {\em J. Comb. Theory Ser. B 102}, 6 (2012), 1295--1320.

\bibitem{fartr1}
{\sc Borodin, O.~V., and Raspaud, A.}
\newblock A sufficient condition for a planar graph to be 3-colorable.
\newblock {\em J. Combin. Theory, Ser.~B 88\/} (2003), 17--27.

\bibitem{le4far}
{\sc Dvo\v{r}\'ak, Z.}
\newblock $3$-choosability of planar graphs with $(\le\!4)$-cycles far apart.
\newblock {\em Journal of Combinatorial Theory, Series B 104\/} (2014), 28--59.

\bibitem{DvoKawTho}
{\sc Dvo\v{r}\'ak, Z., Kawarabayashi, K., and Thomas, R.}
\newblock Three-coloring triangle-free planar graphs in linear time.
\newblock {\em Trans. on Algorithms 7\/} (2011), article no. 41.

\bibitem{dkt}
{\sc Dvo\v{r}\'ak, Z., Kr\'al', D., and Thomas, R.}
\newblock Coloring planar graphs with triangles far apart.
\newblock {\em ArXiv 0911.0885v1\/} (Nov. 2009).

\bibitem{trfree3}
{\sc Dvo\v{r}\'ak, Z., Kr\'al', D., and Thomas, R.}
\newblock Three-coloring triangle-free graphs on surfaces {III}. {G}raphs of
  girth five.
\newblock {\em ArXiv 1402.4710\/} (Feb. 2014).

\bibitem{trfree4}
{\sc Dvo\v{r}\'ak, Z., Kr\'al', D., and Thomas, R.}
\newblock Three-coloring triangle-free graphs on surfaces {IV}. {B}ounding face
  sizes of $4$-critical graphs.
\newblock {\em ArXiv 1404.6356v3\/} (May 2015).

\bibitem{trfree6}
{\sc Dvo\v{r}\'ak, Z., Kr\'al', D., and Thomas, R.}
\newblock Three-coloring triangle-free graphs on surfaces {VI}.
  $3$-colorability of quadrangulations.
\newblock {\em ArXiv 1509.01013\/} (Sept. 2015).

\bibitem{col8cyc}
{\sc Dvo\v{r}\'ak, Z., and Lidick\'y, B.}
\newblock 3-coloring triangle-free planar graphs with a precolored 8-cycle.
\newblock {\em J. Graph Theory 80\/} (2015), 98--111.

\bibitem{cylgen-part2}
{\sc Dvo\v{r}\'ak, Z., and Lidick\'y, B.}
\newblock Fine structure of $4$-critical triangle-free graphs {II}. {P}lanar
  triangle-free graphs with two precolored $4$-cycles.
\newblock {\em SIAM J. Discrete Math. 31\/} (2017), 865--874.

\bibitem{gimbel}
{\sc Gimbel, J., and Thomassen, C.}
\newblock Coloring graphs with fixed genus and girth.
\newblock {\em Trans. Amer. Math. Soc. 349\/} (1997), 4555--4564.

\bibitem{grotzsch1959}
{\sc Gr{\"o}tzsch, H.}
\newblock Ein {D}reifarbensatz f\"{u}r dreikreisfreie {N}etze auf der {K}ugel.
\newblock {\em Math.-Natur. Reihe 8\/} (1959), 109--120.

\bibitem{grunbaum}
{\sc Gr{\"u}nbaum, B.}
\newblock Gr{\"o}tzsch's theorem on 3-colorings.
\newblock {\em Michigan Math. J. 10\/} (1963), 303--310.

\bibitem{conj-havel}
{\sc Havel, I.}
\newblock On a conjecture of {B}. {G}r\"unbaum.
\newblock {\em J. Combin. Theory, Ser.~B 7\/} (1969), 184--186.

\bibitem{havel-zbarv}
{\sc Havel, I.}
\newblock O zbarvitelnosti rovinn\'ych graf\r{u} t\v{r}emi barvami.
\newblock {\em Mathematics (Geometry and Graph Theory)\/} (1970), 89--91.

\bibitem{montasweb}
{\sc Montassier, M.}
\newblock The 3-color problem.
\newblock
  \url{http://www.lirmm.fr/~montassier/index.php?n=Site.ThreeColorProblem}.

\bibitem{PosThoHyperb}
{\sc Postle, L., and Thomas, R.}
\newblock Hyperbolic families and coloring graphs on surfaces.
\newblock {\em Transactions of the American Mathematical Society, Series B 5\/}
  (2018), 167--221.

\bibitem{tw-klein}
{\sc Thomas, R., and Walls, B.}
\newblock Three-coloring {K}lein bottle graphs of girth five.
\newblock {\em J. Combin. Theory, Ser.~B 92\/} (2004), 115--135.

\bibitem{thom-many}
{\sc Thomassen, C.}
\newblock Many $3$-colorings of triangle-free planar graphs.
\newblock {\em J. Combin. Theory, Ser.~B 97\/} (2007), 334--349.

\bibitem{tutteflow}
{\sc Tutte, W.}
\newblock A contribution on the theory of chromatic polynomials.
\newblock {\em Canad. J. Math. 6\/} (1954), 80--91.

\end{thebibliography}

\end{document}